\documentclass[12pt]{amsart}
\usepackage{amsmath,amssymb,amsthm,mathtools}
\usepackage{multicol}
\usepackage{geometry}
\usepackage{enumerate}
\usepackage{bm,longtable}
\usepackage{graphicx}
\usepackage{hyperref}
\hypersetup{
	colorlinks=true,
	linkcolor=cyan,
	citecolor=cyan,
}
\usepackage{exscale}
\usepackage{breqn}
\usepackage{latexsym}
\usepackage{xspace}
\usepackage{array}
\newtheorem{theorem}{Theorem}[section]
\newtheorem{lemma}{Lemma}[section]

\newtheorem{de}{Definition}[section]
\newtheorem{prop}{Proposition}[section]

\newcommand{\G}{\Gamma_{a,b}(2,2)}
\newcommand{\CC}{\mathcal{C}_\nu}

\newcommand{\TT}{\mathrm{Tr}_{\mathbb{F}_{q^n}/\mathbb{F}_q}}
\newcommand{\RK}{$r$-primitive $k$-normal }

\title[On $r$-primitive $k$-normal polynomials with two prescribed coefficients]{On $r$-primitive $k$-normal polynomials with two prescribed coefficients over finite fields}

\author{Avnish K. Sharma}
\address{Department of Mathematics, Shri Ram College of Commerce, University of Delhi, New Delhi-110007, India}
\email{avkush94@gmail.com}

\author{Mamta Rani}
\address{Department of Mathematics, University of Delhi, New Delhi-110007, India}
\email{mamta11singla@gmail.com} 

\author{Sharwan K. Tiwari$^*$}
\address {Cryptography Research Center, Technology Innovation Institute, Abu Dhabi, UAE}
\email{shrawant@gmail.com}
\thanks{$^*$ Corresponding author}

\author{Anupama Panigrahi}
\address{Department of Mathematics, University of Delhi, New Delhi-110007, India}
\email{anupama.panigrahi@gmail.com} 

\keywords{Finite fields; $r$-Primitive elements; $k$-Normal elements; Minimal polynomials; Norm and Trace; Additive and Multiplicative characters.}

\subjclass{12E20, 11T23.}

\begin{document}
	\maketitle
	\sloppy	
		
	\begin{abstract}
		This article investigates the existence of an \RK  polynomial, defined as the minimal polynomial of an $r$-primitive $k$-normal element in $\mathbb{F}_{q^n}$, with a specified degree $n$ and two given coefficients over the finite field $\mathbb{F}_{q}$. Here, $q$ represents an odd prime power, and $n$ is an integer. The article establishes a sufficient condition to ensure the existence of such a polynomial. Using this condition, it is demonstrated that a $2$-primitive $2$-normal polynomial of degree $n$ always exists over $\mathbb{F}_{q}$ when both $q\geq 11$ and $n\geq 15$. However, for the range $10\leq n\leq 14$, uncertainty remains regarding the existence of such a polynomial for $71$ specific pairs of $(q,n)$. Moreover, when $q<11$, the number of uncertain pairs reduces to $16$. Furthermore, for the case of $n=9$, extensive computational power is employed using SageMath software, and it is found that the count of such uncertain pairs reduced to $3988$.
	\end{abstract}

\section{Introduction}\label{sec1}
	In this article, we consider finite fields $\mathbb{F}_{q^n}$, the $n$-th degree extension field of $\mathbb{F}_{q}$, with $q=p^m$, where $p$ is an odd prime and $m,n$ are positive integers. We define an element $\alpha$ of the multiplicative group $\mathbb{F}_{q^n}^*$ as an $r$-primitive element if its multiplicative order is $\frac{q^n-1}{r}$, where $r\mid q^n-1$. If $r=1$, then $r$-primitive elements are exactly the primitive elements over $\mathbb{F}_{q}$. Furthermore, $\alpha$ is called normal over $\mathbb{F}_q$ if its conjugates ${\alpha,\alpha^q,\ldots,\alpha^{q^{n-1}}}$ form a basis of $\mathbb{F}_{q^n}$ as a vector space over $\mathbb{F}_q$. The term $k$-normal element was introduced by \cite{huc13} and refers to an element whose conjugates span a subspace of dimension $n-k$ over $\mathbb{F}_q$. If $k=0$, then the $k$-normal elements are exactly the normal elements over $\mathbb{F}_q$. We call an element \RK if it is $r$-primitive as well as $k$-normal. In addition, the corresponding minimal polynomials of a primitive, normal or primitive normal elements are called primitive, normal or primitive normal polynomials, respectively. 
	
	Extensive research has been conducted on the existence of primitive, normal, or primitive normal elements with specific properties like norm and trace over finite fields. These investigations have been motivated by their applications in various domains, including Coding theory, Cryptography, and signal processing. For a comprehensive understanding, interested readers are encouraged to refer to \cite{Lenstra,cao2014prim,cohen2006strong,gauss,kapetanakis2019variations,Avnish1,Avnish2,Rani1}. The norms and traces of elements correspond to certain coefficients in their minimal polynomials, which has sparked the interest of researchers in exploring special polynomials with predetermined coefficients.
	
	In 1992, Hansen and Mullen proposed a conjecture (Hansen-Mullen conjecture) stating that, for any $n\geq 2$, there exists a primitive polynomial of degree $n$, denoted as $f(x) = x^n - \sigma_1 x^{n-1} + \ldots + (-1)^n \sigma_n$. In this polynomial, the coefficients $\sigma_m$ can take values from the finite field $\mathbb{F}_q$, except for certain cases: $(q, n, m, a) = (q, 2, 1, 0)$, $(4, 3, 1, 0)$, $(4, 3, 2, 0)$, and $(2, 4, 2, 1)$. In 2004, Fan and Han \cite{fan1} introduced the $p$-adic method and demonstrated that the Hansen-Mullen conjecture holds approximately, with the only possible exceptions being when $m$ equals $\frac{n+1}{2}$ if $n$ is odd and when $m$ equals $\frac{n}{2}$ or $\frac{n}{2} + 1$ if $n$ is even. They also established in their work, \cite{fan2}, that for $n \geq 7$, the conjecture holds true when $q$ is an even prime power and $n$ is odd. Later, \cite{cohenprimpoly} proved that the Hansen-Mullen conjecture holds for $n\geq 9$, and in a subsequent paper \cite{cohenhansen}, they extended these results to all $n\geq 2$. Furthermore, \cite{fan3} demonstrated in 2004, that there exists a constant $C(n)$ such that, for any degree $n$, there will always be a primitive polynomial with the prescribed first $m=\lfloor\frac{n-1}{2}\rfloor$ coefficients when $q>C(n)$. These results have also been generalized to primitive normal polynomials, as discussed in various works such as \cite{fan4, fan5, fan6, fan7, morgan}. Motivated by this, in the present paper authors explore the \RK polynomial with two prescribed coefficients over $\mathbb{F}_q$, where \RK polynomial is defined as the minimal polynomial of an \RK element over finite fields. 
	
	For small values of $r$, $r$-primitive elements are also regarded as high order elements. In a previous work by Negre \cite{negre}, a new class of bases termed as quasi-normal bases was introduced. To define quasi-normal bases, he exploits the structure of $\mathbb{F}_{q^n}$ as $\mathbb{F}_q[x]$-module under the action: $f\circ\alpha=\left(\sum_{i=1}^{m}a_ix^i\right)\circ\alpha=\sum_{i=1}^{m}a_i\alpha^{q^i}.$ Under this action, each submodule $M$ of $\mathbb{F}_{q^n}$ satisfies that $\alpha^q\in M$ for each $\alpha\in M.$ A basis of the form $\{\alpha, \alpha^q,\ldots,\alpha^{q^{\mathrm{dim}M-1}}\}$ of $M$ is called a normal basis of $M$. Let $\mathbb{F}_{q^n}$ be a direct sum of submodules $M_i$; $1\leq i\leq m$, then quasi-normal basis associated to this decomposition is the union of the normal bases of each submodule $M_i$ of the decomposition. These bases facilitate simple exponentiation to the power $q$ in $\mathbb{F}_{q^n}$, although not as simple as in normal bases, and could serve as a viable alternative to normal bases, particularly for fields $\mathbb{F}_{q^n}$ lacking normal bases with low complexity. An interesting observation is that	$k$-normal elements generate quasi-normal bases in finite fields. Consequently, $r$-primitive $k$-normal elements may serve as substitutes for primitive normal elements in certain applications. These elements possess high multiplicative order and contribute to the formation of bases with low complexity, with a slight compromise on efficiency. This observation motivates a comprehensive exploration into the existence of $r$-primitive $k$-normal elements and their corresponding minimal polynomials over finite fields.
	
	Recently, authors \cite{Rani2, Rani3} explored the existence of \RK elements and their multiplicative inverses in finite fields. Furthermore, in \cite{Rani4}, authors delve into the investigation of an \RK element possessing a predetermined norm and trace over finite fields. This exploration subsequently provides evidence for the existence of an \RK polynomial with two specific coefficients: the constant coefficient and the coefficient immediately adjacent to the leading term. Building upon this research, in the present paper, authors extend the exploration to examine the existence of a \RK polynomial with two prescribed coefficients: the first two coefficients immediately adjacent to the leading term.
	
	{The outline of the thesis is as follows: Section \ref{Sec2} provide some preliminary results. In Section \ref{Sec3}, we derive a sufficient criterion for the existence of \RK polynomials along with the condition that the first two coefficients immediately adjacent to the leading one are also prescribed. Finally in Section \ref{Sec4}, we demonstrate that when $q\geq 11$ and $n\geq 15$, a $2$-primitive $2$-normal polynomial of degree $n$ with two prescribed coefficients always exists over $\mathbb{F}_{q}$, however, within the range $10\leq n\leq 14$, there are $71$ specific pairs of $(q,n)$ for which the existence of such a polynomial remains uncertain. Additionally, when $q<11$, the number of uncertain pairs is to $16$. Furthermore, for the case of $n=9$, extensive computational power is employed using SageMath software, and it is found that the count of such uncertain pairs reduced to $3988$.}


	\section{Prerequisites}\label{Sec2}
	In this section, we recall some definitions, lemmas and estimates that are needed throughout the article.
	
	Let $e\mid q^n-1$ and $\beta \in \mathbb{F}_{q^n}.$ Then $\beta$ is {\it$e$-free} if and only if $e$ and $\frac{q^n-1}{\mathrm{ord}(\beta)}$ are co-prime, where $\mathrm{ord}(\beta)$ is the multiplicative order of $\beta$ in $\mathbb{F}_{q^n}^*.$ Clearly, an element $\beta \in \mathbb{F}_{q^n}^*$ is primitive if and only if it is $(q^n-1)$-free. 
	
	Let $h(x)=\sum_{i=0}^{m}a_ix^i\in\mathbb{F}_{q}[x]$ with $a_m\neq0$. Then, the additive group $\mathbb{F}_{q^n}$ forms an $\mathbb{F}_{q}[x]$- module under the action $h\circ\beta=\sum_{i=0}^{m}a_i\beta^{q^i}$. Notice that, $(x^n-1)\circ\beta=0$ for all $\beta\in \mathbb{F}_{q^n},$ which leads to the following definition. The {\it$\mathbb{F}_q$-order} of an element $\beta\in \mathbb{F}_{q^n},$ denoted by $\mathrm{Ord}_q(\beta)$, is the least degree monic divisor $h(x)$ of $x^n-1$ such that $h\circ\beta=0$.
	
	Similar to $e$-free elements, we call an element $\beta$ to be {\it$h$-free} if and only if $h(x)$ and $\frac{x^n-1}{\mathrm{Ord}_q(\beta)}$ are co-prime.   Clearly, an element $\beta \in \mathbb{F}_{q^n}^*$ is normal if and only if it is $(x^n-1)$-free. Moreover, from \cite[Theorem 3.2]{huc13}, in terms of the $\mathbb{F}_q$-order, an element $\beta\in \mathbb{F}_{q^n}$ is $k$-normal over $\mathbb{F}_q$ if and only if its $\mathbb{F}_q$-order is of degree $n-k$. The following lemma provides a way to construct $k$-normal elements from a given normal element.
	
	\begin{lemma}{\upshape\cite[Lemma 3.1]{reis19}}\label{L2.1}
		Let $\beta\in\mathbb{F}_{q^n}$ be a normal element over $\mathbb{F}_q$ and $g\in \mathbb{F}_q[x]$ be
		a polynomial of degree $k$ such that $g$ divides $x^n-1$. Then $\alpha=g\circ\beta$ is $k$-normal.
	\end{lemma} 
	
	Now, we recall the characters of a group. Let $\mathfrak{G}$ be a finite abelian group. A {\it character} $\chi$ of $\mathfrak{G}$ is a homomorphism from $\mathfrak{G}$ into the multiplicative group of  complex numbers of unit modulus. The set of all such characters of $\mathfrak{G}$, denoted by $\widehat{\mathfrak{G}}$, forms a multiplicative group and $\mathfrak{G}\cong\widehat{\mathfrak{G}}$. A character $\chi$ is called the trivial character if $\chi(\alpha)=1$ for all $\alpha\in \mathfrak{G}$, otherwise it is a non-trivial character.
	
	\begin{lemma}{\upshape\cite[Theorem 5.4]{Nieder}}\label{L2.2}
		Let $\chi$ be any non-trivial character of a finite abelian group $\mathfrak{G}$ and $\alpha \in \mathfrak{G}$ be any non-trivial element, then
		$\sum_{\alpha \in \mathfrak{G}} \chi(\alpha)=0 \ \text{and} \ \sum_{\chi \in \widehat{\mathfrak{G}}} \chi(\alpha)=0.$
	\end{lemma}
	
	Let $\psi$ denote the additive character for the additive group $\mathbb{F}_{q^n}$, and $\chi$ denote the multiplicative character for the multiplicative group $\mathbb{F}_{q^n}^*.$ The additive character $\psi_0$ defined by $\psi_{0}(\beta)=e^{2\pi i \mathrm{Tr}(\beta)/p}, \ \text{for all} \ \beta \in \mathbb{F}_{q^n},$ where $p$ is the characteristic of $\mathbb{F}_{q^n}$ and $\mathrm{Tr}$ is the absolute trace function from $\mathbb{F}_{q^n}$ to    $\mathbb{F}_p$, is called the {\it canonical additive character} of $\mathbb{F}_{q^n}$. Moreover, every additive character $\psi_\beta$ for $\beta \in \mathbb{F}_{q^n}$ can be expressed in terms of the canonical additive character $\psi_0$ as $\psi_\beta(c)=\psi_{0}(\beta c),\ \text{for all} \ c \in \mathbb{F}_{q^n}.$ 
	
	For any $\psi \in \widehat{\mathbb{F}}_{q^n}$, $\alpha \in \mathbb{F}_{q^n}$ and $g(x)\in \mathbb{F}_q[x]$, $\widehat{\mathbb{F}}_{q^n}$ is an $\mathbb{F}_{q}[x]$-module under the action $\psi\circ g(\alpha)=\psi(g\circ \alpha).$ The {\it $\mathbb{F}_q$-order} of an additive character $\psi\in \widehat{\mathbb{F}}_{q^n}$, denoted by $\mathrm{Ord}_q(\psi)$, is the least degree monic divisor $g(x)$ of $x^n-1$ such that $\psi\circ g$ is the trivial character and there are precisely $\Phi_q(g)$ characters of $\mathbb{F}_q$-order $g(x)$, where $\Phi_q(g)=|(\mathbb{F}_{q}[x]/<g>)^*|$. Moreover, $\sum_{h\mid g}\Phi_q(h)=q^{\mathrm{deg}(g)}.$
	
	We shall need the following lemma to prove our sufficient condition.
	
	\begin{lemma}\label{L2.3}{\upshape\cite[Theorem 2C]{schmidt}}
		Let $f(x)$ be a polynomial over $\mathbb{F}_{q^n}$ with $m$ distinct zeroes. Let $\chi$ be a non-trivial multiplicative character of order $d$ of $\mathbb{F}_{q^n}^*.$ Suppose that $f(x)$ is not of the form $h(x)^{d}$ for any $h(x)\in \mathbb{F}[x],$ where $\mathbb{F}$ is the algebraic closure of $\mathbb{F}_{q^n}.$ Then we have  
		\begin{equation*}
			\Big |\sum_{\substack{\alpha \in \mathbb{F}_{q^n}}}\chi\big(f(\alpha)\big)\Big|\leq (m-1)q^{n/2}.
		\end{equation*}
	\end{lemma}
	
	\begin{lemma}\label{L2.4}{\upshape\cite[Theorem 2G]{schmidt}}
		Let $f(x),\ g(x) \in \mathbb{F}_{q^n}[x]$ be polynomials with $f(x)$ having $m$ distinct zeroes and $g(x)$ having degree $n$. Let $\chi$ be a non-trivial multiplicative character of $\mathbb{F}_{q^n}^*$ and $\psi$ be a non-trivial additive character of $\mathbb{F}_{q^n}.$ Suppose that $g(x)\neq h(x)^{q^n}-h(x)$ in $\mathbb{F}[x],$ where $\mathbb{F}$ is the algebraic closure of $\mathbb{F}_{q^n}.$ Then we have
		\begin{equation*}
			\Big|\sum_{\substack{\alpha \in \mathbb{F}_{q^n}}} \chi\big(f(\alpha)\big)\psi\big(g(\alpha)\big)\Big|\leq (m+n-1)q^{n/2}.
		\end{equation*} 
	\end{lemma}
	
	Now, we recall the characteristic functions for the set of $e$-free and $g$-free elements, where $e\mid q^n-1$ and $g\mid x^n-1$. 
	For the set of $e$-free elements of $\mathbb{F}_{q^n}^*$, we have the following characteristic function $\rho_e:\mathbb{F}_{q^n}^*\to \{0,1\}$
	\begin{equation*}\tag{I}\label{I}
		\rho_{e}(\alpha)= \frac{\phi(e)}{e}\sum_{d|e}\frac{\mu(d)}{\phi(d)}\sum_{\chi_{d}}\chi_{d}(\alpha),
	\end{equation*}
	
	where $\mu$ is the M\"obius function and the internal sum runs over all the multiplicative characters $\chi_d$ of order $d$.  Similar to the characteristic function for the set of $e$-free elements, we have the following characteristic function $\kappa_f:\mathbb{F}_{q^n}\to \{0,1\}$ for the set of $f$-free elements of $\mathbb{F}_{q^n}$
	\begin{equation*}\tag{II}\label{II}
		\kappa_{f}(\alpha)= \dfrac{\Phi_q(f)}{q^{\text{deg}(f)}}\sum_{h|f}\dfrac{\mu'(h)}{\Phi_{q}(h)}\sum_{\psi_{h}}\psi_{h}(\alpha),
	\end{equation*}
	where $\mu'$ is the analogue of the M\"obius function, which is defined as $\mu'(h)=(-1)^t$, if $h(x)$ is a product of $t$ distinct monic irreducible polynomials, otherwise 0, and the internal sum runs over the additive characters $\psi_{h}$ of the $\mathbb{F}_{q}$-order $h(x)$. Along with these, we also need the following characteristic function for the elements of prescribed trace.
	\begin{equation*}\tag{III}\label{III}
		\tau_b(\alpha)=\frac{1}{q}\sum_{\psi \in \widehat{\mathbb{F}_q}}\psi(\mathrm{Tr}_{\mathbb{F}_{q^n}/{\mathbb{F}_q}}(\alpha)-b).
	\end{equation*} 
	We also need the following definition which can be deduced from Lemma \ref{L2.1}.
	
	\begin{de}
		For any $\alpha\in\mathbb{F}_{q^n}$, the character sum
		\begin{equation}\tag{IV}\label{IV}
			I_0(\alpha)=\frac{1}{q^n}\sum_{\Omega \in \widehat{\mathbb{F}}_{q^n}}\Omega(\alpha)
		\end{equation}
		takes the value $1$ if $\alpha=0$, otherwise $0$.
	\end{de}
	The following lemma provides a certain sum, which we will need in proving our sufficient condition.
	\begin{lemma}\label{L2.5}{\upshape\cite[Lemma 2.5]{RK}}
		Let $g\in\mathbb{F}_{q}[x]$ be a divisor of $x^n-1$ of degree $k$ and let $\psi$ and $\Omega$ be additive characters. Then 
		$$\sum_{\beta\in\mathbb{F}_{q^n}}\psi(\beta)\Omega(g\circ\beta)^{-1}=\left\{\begin{array}{cl}
			q^n,& \text{if  } \psi=g\circ\Omega\\
			0,& \text{otherwise}
		\end{array}\right.,$$
		where $g\circ\Omega(\alpha)=\Omega(g\circ\alpha)$ for all $\alpha\in\mathbb{F}_{q^n}$. Furthermore, for a given additive character $\psi$, the set $\{\Omega\in\widehat{\mathbb{F}}_{q^n}
		|\psi=g\circ\Omega\}$ has $q^k$ elements if $\mathrm{Ord}(\psi)\mid\frac{x^n-1}{g}$, otherwise it is empty.			
	\end{lemma}


	\section{General results}\label{Sec3}
	\subsection{Sufficient condition}
	In this subsection, our aim is to establish a sufficient condition to ensure the existence of an \RK polynomial $\sigma(x)=x^n+a_1x^{n-1}+a_2x^{n-2}+\ldots+a_n\in \mathbb{F}_{q}[x]$ with the prescribed coefficients $a_1$ and $a_2$. Since, the degree of the minimal polynomial of an \RK element may not be $n$ always, we will first try to establish a criteria when it is always the case. We know that the conjugates $\alpha,\alpha^q,\alpha^{q^2},\ldots,\alpha^{q^{n-1}}$ of $\alpha\in\mathbb{F}_{q^n}$ over $\mathbb{F}_{q}$ are all distinct if and only if the minimal polynomial of $\alpha$ has degree $n$. Otherwise, the degree $d$ of the minimal polynomial is a proper divisor of $n$ and the conjugates of $\alpha$ are $\alpha,\alpha^{q},\alpha^{q^2},\ldots,\alpha^{q^{d-1}}$, each repeated $n/d$ times. Now, if $\alpha\in\mathbb{F}_{q^n}$ is a $k$-normal element over $\mathbb{F}_{q}$, then the degree $d$ of the minimal polynomial of $\alpha$ will be greater then or equal to $n-k$, and if $n/2<n-k$, then $d$ can never divide $n$. Hence, if $k<n/2$, then $d=n$, i.e. the minimal polynomial of $\alpha$ will be of degree $n$. Therefore, we assume that $k<n/2$ throughout the article so that the degree of the minimal polynomial of $\alpha$ is exactly $n$. Moreover, for an irreducible polynomial of degree $n$, we have the following result by \cite{han96}, which expresses the coefficient $a_2$ in terms of the traces. 
	
	\begin{lemma}{\upshape\cite[Lemma 1]{han96}}\label{L3.1}
		Let $\sigma(x) = x^n + a_1x^{n-1} + a_2x^{n-2}+\ldots + a_n$ be an irreducible polynomial over $\mathbb{F}_{q}$, $\alpha$ be a root of $\sigma(x)$ in $\mathbb{F}_{q^n}$, where $q$ is odd. Then $a_2=\frac{1}{2}\left(\TT(\alpha)^2 - \TT(\alpha^2)\right)$, where $\TT(x)$ is the trace from $\mathbb{F}_{q^n}$ to $\mathbb{F}_q$.	
	\end{lemma}
	Further, we also know that $a_1=-\TT(\alpha)$. Therefore, if $k<n/2$, then the question of the existence of an \RK polynomial $\sigma(x)$ of degree $n$, reduces to the existence of an \RK element $\alpha\in\mathbb{F}_{q^n}$ satisfying the conditions $\TT(\alpha)=a$ and $\TT(\alpha^2)=b$, for any prescribed $a$ and $b$ in $\mathbb{F}_{q}$. Furthermore, using Lemma \ref{L2.1}, and the fact that the trace of an element is non-zero if and only if it is $(x-1)$-free, we will establish the following result. 
	
	\begin{theorem}\label{T3.1}
		Let $q$ be an odd prime power and $n$ be a natural number. Further, let $r$ and $k$ be integers such that $r\mid q^n-1$, and $k<n/2$ such that there exists a $k$ degree polynomial $g(x)\in\mathbb{F}_{q}[x]$. Then there exists an \RK element $\alpha$ satisfying $\TT(\alpha)=a$ and $\TT(\alpha^2)=b$ for any prescribed $a,b\in\mathbb{F}_{q}$ if $$q^{n/2-k-2}>2rW(q^n-1)W\left(\frac{x^n-1}{g}\right).$$
	\end{theorem} 
	
	\begin{proof}
		Instead of proving the existence of an \RK element $\alpha$ with $\TT(\alpha)=a$ and $\TT(\alpha^2)=b$, we first prove that for $l\mid q^n-1$ and $f\mid x^n-1$, the number $N_{r,k,a,b}(l,f)$ of the elements $\gamma$ that are $l$-free, $\gamma^r=g\circ\beta$ for some $f$-free element $\beta\in\mathbb{F}_{q^n}$, $\TT(\gamma^r)=a$ and $\TT(\gamma^{2r})=b$. Then, for $l=q^n-1$ and $f=x^n-1$, we will obtain our desired \RK element $\alpha=\gamma^r$.  
		
		First, let $x-1\nmid g$, then $a\neq0$, and by using the definitions \ref{I}, \ref{II}, \ref{III} and \ref{IV} of the characteristic functions $\rho_{m}$, $\kappa_{f}$, $\tau_a$, $\tau_b$ and $I_0$, the number $N_{r,k,a,b}(l,f)$ will be given as follows.
		\begin{equation*}
			\begin{aligned}
				N_{r,k,a,b}(l,f)=&\sum_{\gamma\in\mathbb{F}_{q^n}^*}\sum_{\beta\in\mathbb{F}_{q^n}}\rho_{l}(\gamma)\kappa_{f}(\beta)I_0(\gamma^r-g\circ\beta)\tau_a(\gamma^r)\tau_b(\gamma^{2r})\\
				=&\mathcal{H}\sum_{\substack{d\mid l,h\mid f}}\frac{\mu(d)\mu'(h)}{\phi(d)\Phi_q(h)}\sum_{\substack{\chi_{d},\psi_{h},\Omega}}\sum_{\substack{\psi,\psi'\in \widehat{\mathbb{F}}_{q}}}\sum_{\gamma\in\mathbb{F}_{q^n}^*}\sum_{\beta\in\mathbb{F}_{q^n}}\chi_{d}(\gamma)\times\\&\psi_{h}(\beta)\Omega(\gamma^r-g\circ\beta)\psi\left(\TT(\gamma^{r})-a\right)\psi'\left(\TT(\gamma^{2r})-b\right)\\
				=&\mathcal{H}\sum_{\substack{d\mid l,h\mid f}}\frac{\mu(d)\mu'(h)}{\phi(d)\Phi_q(h)}\sum_{\substack{\chi_{d},\psi_{h},\Omega}}S({\chi_d,\psi_{h},\Omega}),
			\end{aligned}	
		\end{equation*}
		where $\mathcal{H}=\frac{\phi(l)\Phi_q(f)}{lq^{\mathrm{deg}(f)+n+2}}$ and
		\begin{align*}
			S({\chi_d,\psi_{h},\Omega})=&\sum_{\substack{\psi,\psi'\in \widehat{\mathbb{F}}_{q}}}\sum_{\gamma\in\mathbb{F}_{q^n}^*}\sum_{\beta\in\mathbb{F}_{q^n}}\chi_{d}(\gamma)\psi_{h}(\beta)\Omega(\gamma^r-g\circ\beta)\psi\left(\TT(\gamma^{r})-a\right)\\&\times\psi'\left(\TT(\gamma^{2r})-b\right).
		\end{align*}
		Further, let $\psi_{0}$ be the canonical additive character in $\widehat{\mathbb{F}}_q$. Then
		\begin{align*}
			S({\chi_d,\psi_{h},\Omega})=&\sum_{\substack{c,c'\in{\mathbb{F}}_q}}\psi_0(-c a-c'b)\sum_{\gamma\in\mathbb{F}_{q^n}^*}\sum_{\beta\in\mathbb{F}_{q^n}}\chi_{d}(\gamma)\psi_{h}(\beta)\Omega(\gamma^r-g\circ\beta)\times\\&\psi_0\left(c\TT(\gamma^r)+c'\TT(\gamma^{2r})\right)\\
			=&\sum_{\substack{c,c'\in{\mathbb{F}}_q}}\psi_0(-c a-c'b)\sum_{\gamma\in\mathbb{F}_{q^n}^*}\sum_{\beta\in\mathbb{F}_{q^n}}\chi_{d}(\gamma)\psi_{h}(\beta)\Omega(\gamma^r-g\circ\beta)\times\\&\tilde{\psi_0}(c\gamma^r+c'\gamma^{2r}),
		\end{align*}
		where, $\tilde{\psi_0}=\psi_0\circ\TT$ is the canonical additive character of $\mathbb{F}_{q^n}$. Now, there exists a multiplicative character $\chi_{q^n-1}$ of order $q^n-1$ such that $\chi_{d}=\chi_{q^n-1}^{m}$ for some integer $0\leq m\leq q^n-2$, and there exist elements $\theta,\zeta\in\mathbb{F}_{q^n}$ such that $\psi_{h}(\beta)=\tilde{\psi_{0}}(\theta\beta)$ and $\Omega(\gamma)=\tilde{\psi_{0}}(\zeta\gamma)$. Thus,
		\begin{align*}
			S({\chi_d,\psi_{h},\Omega})=&\sum_{\substack{c,c'\in{\mathbb{F}}_q}}\psi_0(-c a-c'b)\sum_{\gamma\in\mathbb{F}_{q^n}^*}\sum_{\beta\in\mathbb{F}_{q^n}}\chi_{q^n-1}(\gamma^m)\tilde{\psi_0}(\theta\beta)\times\\&\tilde{\psi_0}(\zeta\gamma^r-\zeta g\circ\beta)\tilde{\psi_0}(c\gamma^r+c'\gamma^{2r})\\
			=&\sum_{\substack{c,c'\in{\mathbb{F}}_q}}\psi_0(-c a-c'b)\sum_{\gamma\in\mathbb{F}_{q^n}^*}\chi_{q^n-1}(\gamma^m)\tilde{\psi_0}\left((c+\zeta)\gamma^r+c'\gamma^{2r}\right)\times\\&\sum_{\beta\in\mathbb{F}_{q^n}}\tilde{\psi_0}\left(\theta\beta-\zeta g\circ\beta\right)
		\end{align*}
		Writing $P(x)=x^{m}$, $Q(x)=(c+\zeta)x^r+c'x^{2r}$ and $R(x)=\theta x-\zeta g\circ x$, and substituting in the above equation, we get
		\begin{align*}
			S({\chi_d,\psi_{h},\Omega})=&\sum_{\substack{c,c'\in{\mathbb{F}}_q}}\psi_0(-c a-c'b)\sum_{\gamma\in\mathbb{F}_{q^n}^*}\chi_{q^n-1}(P(\gamma))\tilde{\psi_0}\left(Q(\gamma)\right)\sum_{\beta\in\mathbb{F}_{q^n}}\tilde{\psi_0}\left(R(\beta)\right).
		\end{align*}
		
		Now, before moving further, first notice that
		$$\sum_{\beta\in\mathbb{F}_{q^n}}\tilde{\psi_0}\left(R(\beta)\right)=\sum_{\beta\in\mathbb{F}_{q^n}}\tilde{\psi_0}\left(\theta\beta-\zeta g\circ\beta\right)=\sum_{\beta\in\mathbb{F}_{q^n}}\psi_h(\beta)\Omega(g\circ\beta)^{-1},$$ where $\psi_h(\beta)=\tilde{\psi_{0}}(\theta\beta)$ and $\Omega(\beta)=\tilde{\psi_{0}}(\zeta\beta)$ and from Lemma \ref{L2.5}, we have  $\sum_{\beta\in\mathbb{F}_{q^n}}\psi_h(\beta)\Omega(g\circ\beta)^{-1}=q^n$ for all $\psi_h$ and $\Omega$ satisfying $\psi_h=g\circ\Omega$, otherwise this sum is zero. Moreover, there will be $q^k$ such characters $\Omega$ for a fixed $\psi_h$, provided $h\mid\frac{x^n-1}{g}$.	So, we can take $h$ to be the divisors of $\tilde{f}$ only, where $\tilde{f}:=\mathrm{gcd}\left(f,\frac{x^n-1}{g}\right)$. Secondly, if $\psi_h=g\circ\Omega$, then $h\circ\psi_h(\gamma)=0$ implies $\tilde{\psi_0}(gh\circ\zeta\gamma)=0$ for all $\gamma\in\mathbb{F}_{q^n}$ provided $\zeta\in\mathbb{F}_{q}$. This further implies that $x^n-1\mid gh$, as the $\mathbb{F}_{q}$-order of $\tilde{\psi_0}$ is $x^n-1$ and $gh\circ\tilde{\psi_0}(\gamma)=0$ for all $\gamma\in\mathbb{F}_{q^n}$. We will use these facts later in the estimation of certain quantity.
		
		Now, let $\chi_1$ and $\psi_1$ be the trivial multiplicative and trivial additive characters respectively, and write
		\begin{equation}\label{Eq1}
			N_{r,k,a,b}(l,f)=\mathcal{H}\big(S_1+S_2+S_3+S_4\big),
		\end{equation}
		where $S_1=S(\chi_1,\psi_1,\psi_1)$,\ \ \ \ 
		$S_2=\sum\limits_{d\mid l}\frac{\mu(d)}{\phi(d)}\sum\limits_{\substack{\chi_d,\\\Omega\neq \psi_1}} S(\chi_d,\psi_1,\Omega)$,\\
		
		$$S_3=\sum\limits_{\substack{d\mid l,h\mid \tilde{f}\\d\neq 1 \mathrm{\ or\  }h\neq1}}\frac{\mu(d)\mu'(h)}{\phi(d)\Phi_q(h)}\sum\limits_{\substack{\chi_d,\psi_h}}S(\chi_d,\psi_{h},\psi_1)$$ and $$S_4=\sum\limits_{\substack{d\mid l,h\mid \tilde{f}\\h\neq 1}}\frac{\mu(d)\mu'(h)}{\phi(d)\Phi_q(h)}\sum\limits_{\substack{\chi_{d},\psi_h,\\ \Omega\neq \psi_1}}S(\chi_d,\psi_{h},\Omega).$$
		We will now separately obtain the estimates on $S_1,S_2,S_3$ and $S_4$ using Lemmas \ref{L2.1} and \ref{L2.4}.
		
		For $S_1$, we have $m=0,\theta=0,\zeta=0$, hence
		\begin{align*}
			S_1=&S(\chi_1,\psi_1,\psi_1)\\
			=&\sum_{\substack{c,c'\in{\mathbb{F}}_q}}\psi_0(-c a-c'b)\sum_{\gamma\in\mathbb{F}_{q^n}^*}\chi_{q^n-1}(1)\tilde{\psi_0}\left(c \gamma^r+c'\gamma^{2r}\right)\sum_{\beta\in\mathbb{F}_{q^n}}\tilde{\psi_0}\left(0\right)\\
			=&q^n\sum_{\substack{c,c'\in{\mathbb{F}}_q}}\psi_0(-c a-c'b)\sum_{\gamma\in\mathbb{F}_{q^n}^*}\tilde{\psi_0}(c\gamma^r+c'\gamma^{2r})\\
			=&q^n\sum_{\substack{c,c'\in{\mathbb{F}}_q\\ \text{not both zero}}}\psi_0(-c a-c'b)\sum_{\gamma\in\mathbb{F}_{q^n}^*}\tilde{\psi_0}(c\gamma^r+c'\gamma^{2r})+ q^n(q^n-1).
		\end{align*}
		Now, if at least one of the $c$ and $c'$ is non-zero then $cx^r+c'x^{2r}\neq p(x)^{q^n}-p(x)$ in $\mathbb{F}[x]$. Then, from Lemma \ref{L2.4}, we get
		$\left | S_{1}-q^n(q^n-1)\right |\leq 2r(q-1)^2q^{3n/2},$ which further implies that
		$$S_{1}\geq q^n(q^n-1)-2rq(q-1)q^{3n/2}.$$
		
		Now, we estimate $S_2$. For this, $\psi_h$ is the trivial additive character i.e. $\theta=0$ and $\Omega$ is a non-trivial additive character i.e. $\zeta\neq0$. Therefore, by Lemma \ref{L2.5}, we get
		$$\sum\limits_{\substack{\Omega\in\widehat{\mathbb{F}}_{q^n}\\ \Omega\neq \psi_1}}\sum_{\beta\in{\mathbb{F}}_{q^n}}{\tilde{\psi_0}(\theta\beta-\zeta g\circ \beta)}=(q^k-1)q^n.$$ 
		Thus, 
		\begin{align*}
			S_2=&\sum\limits_{d\mid l}\frac{\mu(d)}{\phi(d)}\sum\limits_{\substack{\chi_d,\\ \Omega\neq \psi_1}} S(\chi_d,\psi_1,\Omega)\\
			=&(q^k-1)q^n\sum\limits_{d\mid q^n-1}\frac{\mu(d)}{\phi(d)}\sum\limits_{\substack{\chi_d,\\ \Omega\neq \psi_1}}\sum_{\substack{c,c'\in{\mathbb{F}}_q}}\psi_0(-c a-c'b)\sum_{\gamma\in\mathbb{F}_{q^n}^*}\chi_{q^n-1}(\gamma^m)\times\\&\tilde{\psi_0}\left((c+\zeta) \gamma^r+c'\gamma^{2r}\right).
		\end{align*}
		Observe that if $c+\zeta=0$, then $\zeta\in\mathbb{F}_{q}$, and hence from the discussion before Equation \eqref{Eq1}, we get $x^n-1\mid g$, a contradiction. Hence, $(c+\zeta)x^r+c'x^{2r}\neq p(x)^{q^n}-p(x)$ in $\mathbb{F}[x]$. Therefore, from Lemma 	\ref{L2.4}, we get 
		$$|S_2|\leq 2r(q^k-1)q^{3n/2+2}W(l).$$
		
		Now, we estimate $S_3.$ For this, we have
		\begin{align*}
			S_3=&\sum\limits_{\substack{d\mid l,h\mid \tilde{f}\\d\neq 1 \mathrm{\ or\  }h\neq1}}\frac{\mu(d)\mu'(h)}{\phi(d)\Phi_q(h)}\sum\limits_{\substack{\chi_d,\psi_h}}S(\chi_d,\psi_{h},\psi_1)\\	=&\sum\limits_{\substack{d\mid l,h\mid \tilde{f}\\d\neq 1 \mathrm{\ or\  }h\neq1}}\frac{\mu(d)\mu'(h)}{\phi(d)\Phi_q(h)}\sum\limits_{\substack{\chi_d,\psi_h}}\sum_{\substack{c,c'\in{\mathbb{F}}_q}}\psi_0(-c a-c'b)\sum_{\gamma\in\mathbb{F}_{q^n}^*}\chi_{q^n-1}(\gamma^m)\times\\&\tilde{\psi_0}\left(c \gamma^r+c'\gamma^{2r}\right)\sum_{\beta\in\mathbb{F}_{q^n}}\tilde{\psi_0}\left(\theta\beta\right)
		\end{align*}
		Notice that if $h\neq 1$, then $\theta\neq 0$, and hence $S_3=0$, and if $h=1$ i.e. $\theta=0$, then $d\neq 1$ i.e. $m\neq 0$. Therefore, 
		\begin{align*}
			S_3=&q^n\sum\limits_{\substack{d\mid l\\d\neq 1}}\frac{\mu(d)}{\phi(d)}\sum\limits_{\substack{\chi_d}}\sum_{\substack{c,c'\in{\mathbb{F}}_q}}\psi_0(-c a-c'b)\sum_{\gamma\in\mathbb{F}_{q^n}^*}\chi_{q^n-1}(\gamma^m)\tilde{\psi_0}\left(c \gamma^r+c'\gamma^{2r}\right)\\
			=&q^n\sum\limits_{\substack{d\mid l\\d\neq 1}}\frac{\mu(d)}{\phi(d)}\sum\limits_{\substack{\chi_d}}\Big\{\sum_{\substack{c,c'\in{\mathbb{F}}_q\\\text{not both zero}}}\psi_0(-c a-c'b)\sum_{\gamma\in\mathbb{F}_{q^n}^*}\chi_{q^n-1}(\gamma^m)\times\\&\tilde{\psi_0}\left(c \gamma^r+c'\gamma^{2r}\right)+\sum_{\gamma\in\mathbb{F}_{q^n}^*}\chi_{q^n-1}(\gamma^m)\Big\}.
		\end{align*}
		Since, $m\neq0$, $\sum_{\gamma\in\mathbb{F}_{q^n}^*}\chi_{q^n-1}(\gamma^m)=0$  by Lemma \ref{L2.1}, and if at least one of the $c$ and $c'$ is non-zero then $cx^r+c'x^{2r}\neq p(x)^{q^n}-p(x)$ in $\mathbb{F}[x]$. Hence applying Lemma \ref{L2.4}, we have 
		$$|S_3|\leq 2rq(q-1)q^{3n/2}(W(l)-1).$$
		
		Finally, we estimate $S_4$. For this we use
		\begin{align*}
			S_4=&\sum\limits_{\substack{d\mid l,h\mid \tilde{f}\\h\neq 1}}\frac{\mu(d)\mu'(h)}{\phi(d)\Phi_q(h)}\sum\limits_{\substack{\chi_{d},\psi_h,\\ \Omega\neq \psi_1}}S(\chi_d,\psi_{h},\Omega)\\
			=&\sum\limits_{\substack{d\mid l,h\mid \tilde{f}\\h\neq 1}}\frac{\mu(d)\mu'(h)}{\phi(d)\Phi_q(h)}\sum\limits_{\substack{\chi_{d},\psi_h,\\\Omega\neq \psi_1}}\sum_{\substack{c,c'\in{\mathbb{F}}_q}}\psi_0(-c a-c'b)\sum_{\gamma\in\mathbb{F}_{q^n}^*}\chi_{q^n-1}(\gamma^m)\times\\&\tilde{\psi_0}\left((c+\zeta) \gamma^r+c'\gamma^{2r}\right)\sum_{\beta\in\mathbb{F}_{q^n}}\tilde{\psi_0}\left(\theta\beta-\zeta g\circ\beta\right).
		\end{align*}
		From the discussion in paragraph before Equation \eqref{Eq1}, we have $\sum_{\beta\in\mathbb{F}_{q^n}}\tilde{\psi_0}\left(\theta\beta-\zeta g\circ\beta\right)=q^n$ if $\psi_h=g\circ \Omega$, otherwise this sum is zero, and this will happen for $q^k$ such $\Omega$ for every $\psi_h$. Further, if $c+\zeta=0$, then $x^n-1\mid gh$, but from Lemma \ref{L2.5}, $h\mid \frac{x^n-1}{g}$. Hence, $x^n-1=gh$, and there is only one such case, namely $h=\tilde{h}:=\frac{x^n-1}{g}$, in which $c+\zeta$ can be zero, otherwise it is non-zero. Therefore, we can write $S_4$ as follows.
		\begin{align*}
			S_4=&q^{n}\sum\limits_{\substack{d\mid l,h\mid \tilde{f}\\h\neq 1,h\neq\tilde{h}}}\frac{\mu(d)\mu'(h)}{\phi(d)\Phi_q(h)}\sum\limits_{\substack{\chi_{d},\psi_h,\\\Omega\neq \psi_1,\psi_{h}=g\circ \Omega}}\sum_{\substack{c,c'\in{\mathbb{F}}_q}}\psi_0(-c a-c'b)\sum_{\gamma\in\mathbb{F}_{q^n}^*}\chi_{q^n-1}(\gamma^m)\times\\&\tilde{\psi_0}\left((c+\zeta) \gamma^r+c'\gamma^{2r}\right)\sum_{\beta\in\mathbb{F}_{q^n}}\tilde{\psi_0}\left(\theta\beta-\zeta g\circ\beta\right)+\sum_{\substack{d\mid q^n-1}}\frac{\mu(d)\mu'(\tilde{h})}{\phi(d)\Phi_q(\tilde{h})}\times\\&\sum\limits_{\substack{\chi_{d},\psi_{\tilde{h}},\\\Omega\neq \psi_1,\psi_{h}=g\circ \Omega}}\Big\{\sum_{\substack{c,c'\in{\mathbb{F}}_q\\c'\neq0}}\psi_0(-c a-c'b)\sum_{\gamma\in\mathbb{F}_{q^n}^*}\chi_{q^n-1}(\gamma^m)\tilde{\psi_0}\left(c'\gamma^{2r}\right)+\\&\sum_{\substack{c\in{\mathbb{F}}_q}}\psi_0(-c a)\sum_{\gamma\in\mathbb{F}_{q^n}^*}\chi_{q^n-1}(\gamma^m)\Big\}.
		\end{align*}
		Note that, if $(x-1)\nmid g$ then $a\neq0$, and hence by Lemma \ref{L2.1}, $\sum_{c\in\mathbb{F}_{q}}\psi_{0}(-ca)=0$. Moreover, $(c+\zeta)x^r+c'x^{2r}\neq p(x)^{q^n}-p(x)$ in $\mathbb{F}[x]$ if $c+\zeta\neq0$, and $c'x^{2r}\neq p(x)^{q^n}-p(x)$ in $\mathbb{F}[x]$ if $c'\neq0$. Therefore, applying Lemma \ref{L2.4}, we get 
		
		$$|S_4|\leq2rq^{3n/2+k+2}W(l)\left(W(\tilde{f})-1\right).$$
		
		Now, from Equation \eqref{Eq1},
		\begin{align*}
			N_{r,k,a,b}(l,f)&\geq \mathcal{H}\big\{q^n(q^n-1)-2rq(q-1)q^{3n/2}-2r(q^k-1)q^{3n/2+2}W(l)-\\&2rq(q-1)q^{3n/2}\left(W(l)-1\right)-2rq^{3n/2+k+2}W(l)\left(W(\tilde{f})-1\right)\big\}	
		\end{align*}
		Clearly, $N_{r,k,a,b}(l,f)$ will be positive if the right side of the above inequality is positive, and the latter will be true if 
		\begin{align*}
			q^n(q^n-1)>&2rq(q-1)q^{3n/2}+2r(q^k-1)q^{3n/2+2}W(l)+2rq(q-1)q^{3n/2}W(l)\\-&2rq(q-1)q^{3n/2}+2rq^{3n/2+k+2}W(l)W(\tilde{f})-2rq^{3n/2+k+2}W(l)\\
			q^{2n}>&2rq^{3n/2+k+2}W(l)W(\tilde{f})+q^n-2rq^{3n/2+1}W(l).
		\end{align*}
		Since $q^n-2rq^{3n/2+1}W(l)<0,$ the above inequality holds if 
		$$q^{n/2-k-2}>2rW(l)W(\tilde{f}).$$
		
		Now, suppose $x-1\mid g$, then $a=\TT(\gamma^r)=0$, therefore we only need to prescribe $b$ and in this case the number $N_{r,k,a,b}$ will be given as 
		$$N_{r,k,0,b}(l,f)=\sum_{\gamma\in\mathbb{F}_{q^n}^*}\sum_{\beta\in\mathbb{F}_{q^n}}\rho_{l}(\gamma)\kappa_{f}(\beta)I_0(\gamma^r-g\circ\beta)\tau_b(\gamma^{2r}).$$
		Following the same steps as for the case $x-1\nmid g$ i.e. $a\neq0$, we obtain that $N_{r,k,0,b}(l,f)$ will be positive if  
		$$q^{n/2-k-1}>2rW(l)W(\tilde{f}).$$
		Therefore in either case, the sufficient condition for the number $N_{r,k,a,b}(l,f)$ to be positive is
		\begin{equation}\label{Eq2}
			q^{n/2-k-2}>2rW(l)W(\tilde{f}).
		\end{equation}
		
		Now, in particular, if we choose $l=q^n-1$ and $f=x^n-1$, then $\tilde{f}=\frac{x^n-1}{g}$, and the sufficient condition for the existence of an \RK element $\alpha=\gamma^r$ such that $\TT(\alpha)=a$ and $\TT(\alpha^2)=b$ is 
		\begin{equation}\label{Eq3}
			q^{n/2-k-2}>2rW(q^n-1)W\left(\frac{x^n-1}{g}\right).
		\end{equation}
		
		In other word, there will always exist an \RK polynomial $\sigma(x)\in\mathbb{F}_{q^n}[x]$ with prescribed coefficient $a_1$ and $a_2$, if Inequality \eqref{Eq3} holds.
	\end{proof}
	
	\subsection{The Sieving inequality}
	The sieving technique, introduced by \cite{CoHucWc,CoHuc} and further modified by \cite{KapeNBT}, to relax Inequality \eqref{Eq3} will be given as follows. The proofs are similar to that of \cite[Lemma 5.1]{Rani2} and \cite[Propostion 3.6]{RK}, hence omitted.
	
	\begin{lemma}\label{L3.2}
		Let $r\mid q^n-1$, and let $k$ be a non-negative integer. Further, let $l$ be a divisor of $q^n-1$ and $\{p_1,p_2,\ldots, p_s\}$ be the set of remaining distinct primes dividing $q^n-1$. Furthermore, let $f$ be the divisor of $x^n-1$ and $\{f_1,f_2,\ldots, f_t\}$ be the remaining distinct irreducible factors of $x^n-1$. Then 
		\begin{align*}
			N_{r,k,a,b}(q^n-1,x^n-1)\geq& \sum_{i=1}^{s}N_{r,k,a,b}(lp_i,f)+\sum_{i=1}^{t}N_{r,k,a,b}(l,ff_i)-\\&(s+t-1)N_{r,k,a,b}(l,f).
		\end{align*}
	\end{lemma}
	
	\begin{prop}\label{P3.1}
		With the notations of the above lemma, define $$\mathcal{D}:=1-\sum_{i=1}^{s}\frac{1}{p_i}-\sum_{i=1}^{t}\frac{1}{q^{\mathrm{deg}(f_i)}} \text{ and } \mathcal{S}:=\frac{s+t-1}{\mathcal{D}}+2.$$ Suppose $\mathcal{D}>0,$ then $N_{r,k,a,b}(q^n-1,x^n-1)>0$, if
		\begin{equation}\label{Eq4}
			q^{n/2-k-2}>2rW(l)W(\tilde{f})\mathcal{S}.
		\end{equation}
	\end{prop}
	
	\section{Existence of $2$-primitive $2$-normal polynomials with two prescribed traces}\label{Sec4}
	In this section, we aim to identify the specific finite fields where we can find $2$-primitive $2$-normal polynomials $\sigma(x)=x^{n}-a_1x^{n-1}+a_2x^{n-2}+\ldots+a_n$, with predetermined coefficients $a_1$ and $a_2$. Essentially, we will search for the finite fields where we can guarantee the existence of a $2$-primitive $2$-normal element $\alpha$ with $\TT(\alpha)=a$ and $\TT(\alpha^2)=b$, for any given $a$ and $b$ belonging to the field $\mathbb{F}_{q}$. To simplify matters, let $\Gamma_{a,b}(r,k)$ represent the set of all pairs $(q,n)$ where there exists an \RK element $\alpha$ in $\mathbb{F}_{q^n}$ over $\mathbb{F}_{q}$, satisfying $\TT(\alpha)=a$ and $\TT(\alpha^2)=b$.
	
	First, it is important to note that we have a necessary condition for the existence of $2$-primitive $2$-normal elements in $\mathbb{F}_{q^n}$ over $\mathbb{F}_{q}$.
	
	\begin{lemma}\label{L4.1}{\upshape\cite[Lemma 6.1]{Rani2}}
		If there exists a $2$-primitive $2$-normal element in $\mathbb{F}_{q^n}$ over $\mathbb{F}_q$, then $q$ is an odd prime power, $n\geq4$ and $\mathrm{gcd}(q^3-q,n)\neq1$. 
	\end{lemma}
	
	Before moving ahead, we state below two lemmas which provide the bounds on the number of square-free divisors $W(x^n-1)$ and $W(M)$ of the polynomial $x^n-1$ and integer $M$. 
	\begin{lemma}\label{L4.2}{\upshape\cite[Lemma 2.9]{Lenstra}}
		Let $q$ be a prime power and $n$ be a positive integer. Then, we have $W(x^n-1)\leq2^{\frac{1}{2} (n+\mathrm{gcd}(n,q-1))}$. In particular, $W(x^n-1)\leq 2^n$ and $W(x^n-1)=2^n$ if and only if $n\mid(q-1)$. Furthermore, $W(x^n-1) \leq 2^{3n/4}$ if $n\nmid(q-1)$, since in this case, $\mathrm{gcd}(n,q-1)\leq n/2.$
	\end{lemma}
	
	\begin{lemma}\label{L4.3}{\upshape\cite[Lemma 3.7]{CoHuc}}
		For any $M\in \mathbb{N}$ and a positive real number $\nu$, $W(M)\leq \CC M^{1/\nu}$, where $\CC=\prod\limits_{i=1}^{t}\frac{2}{p_i^{1/\nu}}$ and $p_1,p_2,\ldots, p_t$ are the primes $\leq 2^\nu $ that divide $M.$
	\end{lemma}
	We will use the following lemma in the specific case $q=3$.
	\begin{lemma}\label{L4.4}{\upshape\cite[Lemma 2.11]{Lenstra}}
		Let $q=3$ and $n$ be a positive integer. Then $W(x^n-1)\leq 2^{ {n}/{3}+{4}/{3}}$ if $n\neq 4, 8 , 16.$
	\end{lemma}
	
	Assuming $r=2$ and $k=2$, the inequalities represented by \eqref{Eq3} and \eqref{Eq4} can be rewritten as follows:
	
	\begin{equation}\label{Eq5}
		q^{n/2-4} > 4 W(q^n-1) W\left(\frac{x^n-1}{g}\right),
	\end{equation}
	
	\begin{equation}\label{Eq6}
		q^{n/2-4} > 4W(l)W(\tilde{f})\mathcal{S}.
	\end{equation}
	However, it is evident that these inequalities are never satisfied for $n<9$. Henceforth, we consider only cases where $n\geq 9$, and our objective is to explicitly determine the finite fields $\mathbb{F}_{q^n}$ that satisfy the condition $(q,n)\in\G$.
	
	\subsection{Cases $q<11$.}
	
	In the next lemma, we deal with the cases $q=3,5,7,9$ and $n\geq 9$
	
	\begin{lemma}\label{L4.5}
		Let $n\geq9$ be an integer and $q=3,5,7,9$ be such that $\mathrm{gcd}(q^3-q,n)\neq1.$ Then $(q,n)\in\G$ except possibly for the following pairs.
		\begin{enumerate}[$(i)$]
			\item $q=3$ and $n= 9, 10, 12, 14, 15, 16, 18.$
			\item $q=5,7,9$ and $n=9,10,12.$
		\end{enumerate}
	\end{lemma}
	\begin{proof}
		\noindent
		To begin, consider the case where $q=3$. If $n$ does not equal 4, 8, or 16, we can deduce from Lemmas \ref{L4.3} and \ref{L4.4} that $W(q^n-1)<\CC q^{n/\nu}$ and $W(x^n-1)\leq 2^{ {n}/{3}+{4}/{3}}$. Consequently, Inequality \eqref{Eq5} is satisfied if
		
		\begin{equation*}
			n > \frac{\log(2^{10/3}\CC q^4)}{(1/2-1/\nu)\log(q)-(\mathrm{log}\hspace{.5mm}2)/3}.
		\end{equation*}
		Choosing $\nu=6.6$, this inequality holds for $n\geq 75$. By testing Inequalities \eqref{Eq5} and \eqref{Eq6}  for the remaining values of $9\leq n<75$, we can conclude that $(q,n)\in\G$, except in the cases where $q=3$ and $n=$ 9, 10, 12, 14, 15, 16, 18.
		
		Now, let $q=5,7,9$, then $n\nmid q-1$, since $n\geq 9$. From Lemma \ref{L4.2}, we have $W(x^n-1)\leq 2^{3n/4}$, therefore Inequality \eqref{Eq5} holds if
		\begin{equation*}
			n>\frac{\log(4\CC q^4)}{(1/2-1/\nu)\mathrm{log}\hspace{.5mm}q-(3\hspace{.5mm}\mathrm{log}\hspace{.5mm}2)/4}.
		\end{equation*}
		Clearly, for $q=5$, choose $\nu=8.4$, then the above inequality holds for $n\geq 192,$ for $q=7$, choose $\nu=7.5$, then it holds for $n\geq 83,$ and for $q=9$, choose $\nu=7.2$, then it holds for $n\geq 60.$  Now, for $q=5$, $9\leq n< 192$, $q=7$, $9\leq n< 83$, and $q=9$, $9\leq n< 60$,  we test Inequality \eqref{Eq5} or \eqref{Eq6}, and get that $(q,n)\in\G$ except possibly for $q=5$ and $n=$ 9, 10, 12, $q=7$ and $n=$ 9, 10, 12, and $q=9$ and $n=$ 9, 10, 12. 
	\end{proof}	
	
	\subsection{The cases $q\geq 11$.}
	
	The next lemma discusses the cases $n\geq 15$ and $q\geq 11$.  
	
	\begin{lemma}\label{L4.6}
		Let $n\geq 15$ be a natural number and $q\geq 11$ be an odd prime power such that $\gcd(q^3-q,n)\neq 1$, then $(q,n)\in\G$.
	\end{lemma}
	\begin{proof}
		Let $q\geq 11$ and $n\geq 15$. Then, using Lemmas \ref{L4.3} and \ref{L4.4}, Inequality \eqref{Eq5} holds if  
		$q^{n/2-4}>4\CC q^{n/\nu}2^{n-2}$, or equivalently
		\begin{equation}\label{Eq7}
			q>(2^n\CC)^{\frac{2\nu}{(n-8)\nu-2n}},
		\end{equation} 
		which is valid if $\nu>\frac{2n}{n-8}$.
		Using SageMath (\cite{sagemath}), we establish the validity of Inequality \eqref{Eq7} for the ranges of $q$ and $n$ given in Table \ref{T1}; that is, for these values, $(q, n)\in\G$.
		
		\begin{table}[h!]
			\centering
			\caption{Pairs $(q,n)\in\G$.}\label{T1}
			{\small
				\begin{tabular}{m{4cm}m{4cm}}
					\hline
					$\nu$& $(q,n)$\\
					\hline
					$\nu=7.9$& $q\geq 11$ and $n\geq 88$ \\
					$\nu=7.3$& $q\geq18$ and $n\geq 50$ \\
					$\nu=7.2$& $q\geq26$ and $n\geq 40$ \\	
					$\nu=7.1$& $q\geq103$ and $n\geq 25$ \\
					$\nu=7.1$& $q\geq472$ and $n\geq 20$ \\
					$\nu=7.3$& $q\geq4015$ and $n\geq 17$ \\
					$\nu=7.4$& $q\geq14356$ and $n\geq 16$ \\
					$\nu=7.6$& $q\geq93903$ and $n\geq 15$ \\
					\hline
			\end{tabular}}
		\end{table}
		The values of $q$ and $n$ which do not satisfy Inequality \eqref{Eq7} are listed in Table \ref{T2}.
		\begin{table}[h!]
			\centering
			\caption{Pairs $(q,n)$ which fail to satisfy Inequality \eqref{Eq7}.}\label{T2}
			{\small
				\begin{tabular}{m{4cm}m{4cm}}
					\hline
					$q$& $n$\\
					\hline
					$11\leq q< 18$& $50\leq n<88$\\
					$11\leq q< 26$& $40\leq n<50$\\
					$11\leq q< 103$& $25\leq n<40$\\
					$11\leq q< 472$& $20\leq n<25$\\
					$11\leq q< 4015$& $17\leq n<20$\\
					$11\leq q< 14356$& $16\leq n<17$\\
					$11\leq q< 93903$& $15\leq n<16$\\
					\hline
			\end{tabular}}
		\end{table}
		We enumerate all the prime powers $q$ within the specified range for the remaining pairs in Table \ref{T2}. Additionally, for each $n$ within the corresponding range, we initially verify Inequality \eqref{Eq5}. If it fails to hold, we proceed to test the sieving inequality (Inequality \eqref{Eq6}) using appropriate selections of $l$ and $f$. Consequently, we have determined that all $(q, n)$ pairs belong to $\G$. 
	\end{proof}
	
	Now, using Inequality \eqref{Eq7}, we obtained that the lower bounds to $q$ in the cases $10\leq n \leq 14$, for which $(q,n)\in\G$ are very large (listed below):
	\begin{enumerate}[$(i)$]
		\item For $n=14$, $q\geq 1.90\times 10^{6}$ (taking $\nu=7.8$). 
		\item For $n=13$, $q\geq 4.34\times 10^{8}$ (taking $\nu=8.1$).
		\item For $n=12$, $q\geq 7.02\times 10^{13}$ (taking $\nu=8.6$).
		\item For $n=11$, $q\geq 3.35\times 10^{30}$ (taking $\nu=9.7$).
		\item For $n=10$, $q\geq 8.34\times 10^{161}$ (taking $\nu=11.9$).
		
	\end{enumerate}
	For the cases that remain, it is not feasible to provide an exhaustive list of all the prime powers. Hence, in the subsequent lemmas, our approach involves initially lowering these lower bounds. Then, within the reduced range, we proceed to list all the prime powers for testing the validity of pairs $(q,n)\in\G$ using Inequality \eqref{Eq5} and the sieving inequality (Inequality \eqref{Eq6}).

	\begin{lemma}\label{L4.7}
		Let $n=10, 14$ and $q\geq 11$ be an odd prime power satisfying $\gcd(q^3-q,n)\neq 1$, then
		\begin{enumerate}[$(i)$]
			\item $(q,10)\in\G$ except possibly for $q=11,$ $13,$ $17,$ $19,$ $23,$ $25,$ $27,$ $29,$ $31,$ $37,$ $41,$ $43,$ $47,$ $49,$ $53,$ $59,$ $61,$ $67,$ $71,$ $73,$ $79,$ $81,$ $83,$ $89,$ $97,$ $101,$ $103,$ $107,$ $109,$ $113,$ $121,$ $125,$ $127,$ $131,$ $137,$ $139,$ $149,$ $151,$ $167,$ $169,$ $179,$ $181,$ $191,$ $199,$ $211,$ $229,$ $239,$ $241,$ $251,$ $271,$ $281,$ $289,$ $311,$ $331,$ $349,$ $361,$ $379,$ $401,$ $421,$ $461,$ $491.$
			
			\item $(q,14)\in\G$ except possibly for $q=29.$
		\end{enumerate}
	\end{lemma}
	\begin{proof}
		Let $q$ be a prime power satisfying $\gcd(q^3-q,n)\neq 1$. Notice that in both the cases $n=10,14$, $\gcd(q^3-q,n)$ is either $2$ or $n$.
		
		\vspace{2mm}
		\noindent
		\textbf{Case $\bm{n=10}$}
		
		Firstly, if we assume that $\gcd(q^3-q,10)=2$, we can conclude that $q\equiv\pm 2(\mathrm{mod}\ 5)$. Consequently, $\gcd(q^2-1,q^8+q^6+q^4+q^2+1)=1$. Additionally, it is worth noting that the primes dividing $q^8+q^6+q^4+q^2+1$ can be expressed in the form $5j+1$. This is due to the fact that if a prime $p$ divides $q^8+q^6+q^4+q^2+1$, then $q^{10}\equiv 1(\mathrm{mod}\ p)$, and the multiplicative order of $q$ modulo $p$ (which divides $\phi(p)=p-1$) must be a divisor of $10$. Thus, we can determine that $p-1$ is divisible by either $1$, $2$, $5$, or $10$. It is important to note that if $1$ or $2$ divides $p-1$, then $\gcd(q^2-1,q^8+q^6+q^4+q^2+1)\neq1$. Therefore, we can conclude that either $5$ divides $p-1$ or $10$ divides $p-1$, which implies that $p$ can be expressed as $5j+1$, where $j$ is an integer.
		
		Secondly, assuming that $\gcd(q^3-q,10)=10$, we can deduce that $5$ divides $q^3-q$. Consequently, $q\equiv0,\pm 1(\mathrm{mod}\ 5)$. In this particular case, $\gcd(q^2-1,q^8+q^6+q^4+q^2+1)$ can either be $1$ or $5$. Furthermore, the primes that divide $q^8+q^6+q^4+q^2+1$ are either $5$ or can be expressed in the form $5j+1$. Therefore, regardless of the case, we can determine that the primes dividing $q^8+q^6+q^4+q^2+1$ are either $5$ or can be expressed as $5j+1$.
		
		Now, we use $l=q^2-1$, and $f=1$ so that $\tilde{f}= \gcd\left(f,\frac{x^{10}-1}{g}\right)=1$ and $t\leq 10$ in Proposition \ref{P3.1}. Let $S_s$ and $P_s$ be, respectively, the sum of the inverses and the product of the first $s$ primes of the form $5j +1$. Since $\{ p_1, \ldots, p_s \}$ is a set of prime numbers which divide $q^8+q^6+q^4+q^2+1$, then for $q<8.34\times 10^{161}$, $P_s \leq q^8+q^6+q^4+q^2+1<5.68\times 10^{1297}$, therefore $s \leq362$ and $S_s<0.5519$. If we suppose $q >10^{6}$, then $\mathcal{D}\geq 1-S_s-10/q=0.4481>0 \text{ and } \mathcal{S}\leq\frac{s+9}{\mathcal{D}}+2=830.11$. Using Inequality \eqref{Eq6}, we get that $(q,n)\in\G$, if $q>(4 \CC \mathcal{S})^{\nu/(\nu-2)}$, and for $\nu=5.4$, this inequality is satisfied for $q\geq 3.533\times 10^{7}$. Now, repeating this process several times for $q<3.533\times 10^{7}$, we can reduce this bound to $496197$. Now, for every prime power $11\leq q< 496197$ and $n=10$, we test Inequality \eqref{Eq5} or \eqref{Eq6} and found that $(q,n)\in\G$ except possibly for the values of $q$ listed in Lemma \ref{L4.7}.
		
		\vspace{2mm}
		\noindent
		\textbf{Case $\bm{n=14}$}
		
		Similar to the previous case, we observe that $\gcd(q^3-q, 14)$ can only be either $2$ or $14$. Additionally, the primes dividing $q^{12}+q^{10}+q^8+q^6+q^4+q^2+1$ are restricted to being either $7$ or of the form $7j+1$ where $j\in\mathbb{Z}$. Following the same method as used for the case when $n=10$, we successfully lower the bound from $1.90\times 10^{6}$ to $30$ for the situation when $n=14$ as well. Subsequently, for each prime power $11\leq q< 30$ and $n=14$, we proceed to test Inequality \eqref{Eq5} or \eqref{Eq6}. As a result of this process, we have confirmed that $(q,n)\in\G$ for all except possibly the case where $q=29.$ 
	\end{proof}
	
	\begin{lemma}\label{L4.8}
		Let $n=11,13$ and $q\geq 11$ be an odd prime power satisfying $\gcd(q^3-q,n)\neq 1$, then
		\begin{enumerate}[$(i)$]
			\item $(q,11)\in\G$ except possibly for $q=23, 43$.
			\item $(q,13)\in\G$ for all  $q\geq 11.$
		\end{enumerate}
	\end{lemma}
	\begin{proof}
		Similar to the previous lemma, we have $\gcd\left(q-1,\frac{q^n-1}{q-1}\right)$ is either $1$ or $n$ for both $n=11$ and $13$, and the primes dividing $\frac{q^n-1}{q-1}$ are either $n$ or of the form $nj+1$, for $j\in\mathbb{Z}$. Now, for $l=q-1$ and $f=1$ we use sieving inequality. Applying the procedure for $l=q-1$ and $f=1$ as explained in the previous lemma, we reduce the lower bound to $q$ from $3.35\times 10^{30}$ and $4.34\times 10^{8}$ to $142$ (for $n=11$) and $22$ (for $n=13$) respectively. We list the remaining prime powers and test Inequality \eqref{Eq5} or \eqref{Eq6}, and get that $(q,n)\in\G$ except possibly for $(23,11)$, $(43,11)$.
	\end{proof}
	
	\begin{lemma}\label{L4.9}
		Let $n=12$ and $q$ be an odd prime power satisfying $\gcd(q^3-q,n)\neq 1$, then $(q,12)\in\G$ except possibly for $q=$ $11$, $13$, $17$, $19$, $23$, $25$, $37$.
	\end{lemma}
	\begin{proof}
		Note that $\gcd(q^3-q,12)$ is always $12$, and the primes dividing $\left(\frac{q^{12}-1}{q^4-1}\right)$ are either $3$ or of the form $3j+1$ for $j\in\mathbb{Z}$. Following the same procedure as in the previous lemmas, we reduce the lower bound $7.02\times 10^{13}$ to $1405$, and then for every prime power $11\leq q\leq 1405$ and $n=12$, we test Inequality \eqref{Eq5} or \eqref{Eq6} and found that $(q,n)\in\G$ except possibly $q=11, 13, 17, 19, 23, 25, 37.$ 
	\end{proof}
	
	\begin{lemma}\label{L4.10}
		Let $n=9$ and $q\geq 11$ be an odd prime power satisfying $\gcd(q^3-q,n)\neq 1$, then $(q,9)\in\G$ except possibly for $3988$ values of $q$ listed in {\upshape Table \ref{T3}}.
	\end{lemma}
	\begin{proof}
		In this case, employing Inequality \eqref{Eq7} with $\nu=19.6$, we establish a lower bound on $q$ as $1.32\times 10^{37982}$. Additionally, when $\gcd(q^3-q,9)\neq 1$, we observe that $q\equiv 0, \pm1(\mathrm{mod}\ 3)$. Moreover, we find that $\gcd(q^3-1,q^6+q^3+1)$ can either be $1$ or $3$, and the primes that divide $q^6+q^3+1$ are either $3$ or of the form $9j+1$ for $j\in\mathbb{Z}$. Therefore, for $l=q^3-1$ and $f=1$, using Proposition \ref{P3.1} and the procedure followed in the previous lemmas, we reduced this bound on $q$ from $1.32\times 10^{37982}$ to $1.26\times 10^{51}$. Now, for $q<1.26\times 10^{51}$, we will use $l=q-1$ and $f=1$ in Proposition \ref{P3.1}. In this case, the primes dividing $\frac{q^9-1}{q-1}$ which are coprime to $q-1$ are either $3$ or of the form $3j+1$; $j\in\mathbb{Z}$. Therefore, again using the previous procedure, we reduce this bound from $1.26\times 10^{51}$ to $6.56\times 10^{14}$. 
		
		To further reduce this bound, we consider another approach. We take $l=\gcd\left(q^9-1,3\cdot5\cdot7\right)$, then $W(l)\leq 8$, and with $f=1$ in Proposition \ref{P3.1}, this bound can be reduced from $6.56\times 10^{14}$ to $3.31\times 10^{7}$. Now, for the remaining range $11\leq q<3.31\times 10^{7}$, we list all the primes powers and individually test Inequality \eqref{Eq5} or \eqref{Eq6}. We obtain that $(q,n)\in\G$ except possibly for the values of $q$  listed in Table \ref{T3}.
	\end{proof}
	
	\section{Acknowledgements}
	The financial support for this research was provided by the Council of Scientific \& Industrial Research under Grant F. No. 09/045(1674)/2019-EMR-I.

\newpage	
	\appendix
	\section{}\label{A1}
	{\tiny
		\begin{longtable}{|p{15cm}|}
				\caption{Prime powers $q$ for which $(q,9)$ for which $(q,9)$ may not be in $\Gamma_{a,b}(2,2)$.\label{T3}}\\
				\hline
				$q$\\
				
				\hline
				\endfirsthead
				\multicolumn{1}{l}{Continuation of Table \ref{T3}}\\ 
				
				\hline
				$q$\\
				\hline
				\endhead

				\hline
				\endfoot
				
				\hline
				\endfoot
				11, 13, 17, 19, 23, 25, 27, 29, 31, 37, 41, 43, 47, 49, 53, 59, 61, 67, 71, 73, 79, 81, 83, 89, 97, 101, 103, 107, 109, 113, 121, 125, 127, 131, 137, 139, 149, 151, 157, 163, 167, 169, 173, 179, 181, 191, 193, 197, 199, 211, 223, 227, 229, 233, 239, 241, 243, 251, 257, 263, 269, 271, 277, 281, 283, 289, 293, 307, 311, 313, 317, 331, 337, 343, 347, 349, 353, 359, 361, 367, 373, 379, 383, 389, 397, 401, 409, 419, 421, 431, 433, 439, 443, 449, 457, 461, 463, 467, 479, 487, 491, 499, 503, 509, 521, 523, 529, 541, 547, 557, 563, 569, 571, 577, 587, 593, 599, 601, 607, 613, 617, 619, 625, 631, 641, 643, 647, 653, 659, 661, 673, 677, 683, 691, 701, 709, 719, 727, 729, 733, 739, 743, 751, 757, 761, 769, 773, 787, 797, 809, 811, 821, 823, 827, 829, 839, 841, 853, 857, 859, 863, 877, 881, 883, 887, 907, 911, 919, 929, 937, 941, 947, 953, 961, 967, 971, 977, 983, 991, 997, 1009, 1013, 1019, 1021, 1031, 1033, 1039, 1049, 1051, 1061, 1063, 1069, 1087, 1091, 1093, 1097, 1103, 1109, 1117, 1123, 1129, 1151, 1153, 1163, 1171, 1181, 1187, 1193, 1201, 1213, 1217, 1223, 1229, 1231, 1237, 1249, 1259, 1277, 1279, 1283, 1289, 1291, 1297, 1301, 1303, 1307, 1319, 1321, 1327, 1331, 1361, 1367, 1369, 1373, 1381, 1399, 1409, 1423, 1427, 1429, 1433, 1439, 1447, 1451, 1453, 1459, 1471, 1481, 1483, 1487, 1489, 1493, 1499, 1511, 1523, 1531, 1543, 1549, 1553, 1559, 1567, 1571, 1579, 1583, 1597, 1601, 1607, 1609, 1613, 1619, 1621, 1627, 1637, 1657, 1663, 1667, 1669, 1681, 1693, 1697, 1699, 1709, 1721, 1723, 1733, 1741, 1747, 1753, 1759, 1777, 1783, 1787, 1789, 1801, 1811, 1823, 1831, 1847, 1849, 1861, 1867, 1871, 1873, 1877, 1879, 1889, 1901, 1907, 1913, 1931, 1933, 1949, 1951, 1973, 1979, 1987, 1993, 1997, 1999, 2003, 2011, 2017, 2027, 2029, 2039, 2053, 2063, 2069, 2081, 2083, 2087, 2089, 2099, 2111, 2113, 2129, 2131, 2137, 2141, 2143, 2153, 2161, 2179, 2187, 2197, 2203, 2207, 2209, 2213, 2221, 2237, 2239, 2243, 2251, 2267, 2269, 2273, 2281, 2287, 2293, 2297, 2309, 2311, 2333, 2339, 2341, 2347, 2351, 2357, 2371, 2377, 2381, 2383, 2389, 2393, 2399, 2401, 2411, 2417, 2423, 2437, 2441, 2447, 2459, 2467, 2473, 2477, 2503, 2521, 2531, 2539, 2543, 2549, 2551, 2557, 2579, 2591, 2593, 2609, 2617, 2621, 2633, 2647, 2657, 2659, 2663, 2671, 2677, 2683, 2687, 2689, 2693, 2699, 2707, 2711, 2713, 2719, 2729, 2731, 2741, 2749, 2753, 2767, 2777, 2789, 2791, 2797, 2801, 2803, 2809, 2819, 2833, 2837, 2843, 2851, 2857, 2861, 2879, 2887, 2897, 2903, 2909, 2917, 2927, 2939, 2953, 2957, 2963, 2969, 2971, 2999, 3001, 3011, 3019, 3023, 3037, 3041, 3049, 3061, 3067, 3079, 3083, 3089, 3109, 3119, 3121, 3125, 3137, 3163, 3167, 3169, 3181, 3187, 3191, 3203, 3209, 3217, 3221, 3229, 3251, 3253, 3257, 3259, 3271, 3299, 3301, 3307, 3313, 3319, 3323, 3329, 3331, 3343, 3347, 3359, 3361, 3371, 3373, 3389, 3391, 3407, 3413, 3433, 3449, 3457, 3461, 3463, 3467, 3469, 3481, 3491, 3499, 3511, 3517, 3527, 3529, 3533, 3539, 3541, 3547, 3557, 3559, 3571, 3581, 3583, 3593, 3607, 3613, 3617, 3623, 3631, 3637, 3643, 3659, 3671, 3673, 3677, 3691, 3697, 3701, 3709, 3719, 3721, 3727, 3733, 3739, 3761, 3767, 3769, 3779, 3793, 3797, 3803, 3821, 3823, 3833, 3847, 3851, 3853, 3863, 3877, 3881, 3889, 3907, 3911, 3917, 3919, 3923, 3929, 3931, 3943, 3947, 3967, 3989, 4001, 4003, 4007, 4013, 4019, 4021, 4027, 4049, 4051, 4057, 4073, 4079, 4091, 4093, 4099, 4111, 4127, 4129, 4133, 4139, 4153, 4157, 4159, 4177, 4201, 4211, 4217, 4219, 4229, 4231, 4241, 4243, 4253, 4259, 4261, 4271, 4273, 4283, 4289, 4297, 4327, 4337, 4339, 4349, 4357, 4363, 4373, 4391, 4397, 4409, 4421, 4423, 4441, 4447, 4451, 4457, 4463, 4481, 4483, 4489, 4493, 4507, 4513, 4517, 4519, 4523, 4547, 4549, 4561, 4567, 4583, 4591, 4597, 4603, 4621, 4637, 4639, 4643, 4649, 4651, 4657, 4663, 4673, 4679, 4691, 4703, 4721, 4723, 4729, 4733, 4751, 4759, 4783, 4787, 4789, 4793, 4799, 4801, 4813, 4817, 4831, 4861, 4871, 4877, 4889, 4903, 4909, 4913, 4931, 4933, 4937, 4943, 4951, 4957, 4967, 4969, 4987, 4993, 4999, 5003, 5009, 5011, 5021, 5023, 5039, 5041, 5051, 5059, 5077, 5081, 5099, 5101, 5107, 5113, 5119, 5147, 5153, 5167, 5171, 5179, 5189, 5197, 5209, 5227, 5231, 5233, 5237, 5261, 5273, 5279, 5281, 5303, 5309, 5323, 5329, 5333, 5347, 5351, 5381, 5387, 5393, 5399, 5407, 5413, 5417, 5419, 5431, 5437, 5441, 5443, 5449, 5471, 5477, 5479, 5483, 5503, 5507, 5519, 5521, 5527, 5531, 5557, 5563, 5569, 5573, 5581, 5591, 5623, 5639, 5641, 5647, 5651, 5653, 5657, 5659, 5669, 5683, 5689, 5701, 5711, 5737, 5741, 5743, 5749, 5779, 5783, 5791, 5801, 5807, 5813, 5821, 5827, 5839, 5843, 5849, 5851, 5857, 5861, 5867, 5869, 5879, 5881, 5897, 5903, 5923, 5927, 5939, 5953, 5981, 5987, 6007, 6011, 6029, 6037, 6043, 6047, 6053, 6067, 6073, 6079, 6091, 6101, 6113, 6121, 6131, 6133, 6143, 6151, 6163, 6173, 6197, 6199, 6203, 6211, 6217, 6221, 6229, 6241, 6247, 6257, 6263, 6269, 6271, 6277, 6287, 6299, 6301, 6311, 6317, 6323, 6329, 6337, 6343, 6353, 6361, 6367, 6373, 6379, 6397, 6421, 6427, 6449, 6451, 6469, 6481, 6491, 6521, 6529, 6547, 6551, 6553, 6561, 6563, 6571, 6577, 6581, 6607, 6619, 6637, 6659, 6661, 6673, 6679, 6689, 6691, 6701, 6703, 6709, 6733, 6737, 6761, 6763, 6781, 6791, 6793, 6803, 6823, 6827, 6829, 6833, 6841, 6857, 6859, 6863, 6869, 6871, 6883, 6889, 6899, 6907, 6911, 6917, 6947, 6949, 6959, 6961, 6967, 6971, 6983, 6991, 6997, 7001, 7019, 7027, 7039, 7043, 7057, 7069, 7079, 7103, 7109, 7121, 7127, 7129, 7151, 7159, 7177, 7187, 7193, 7207, 7211, 7213, 7219, 7229, 7237, 7243, 7283, 7297, 7307, 7309, 7321, 7331, 7333, 7351, 7369, 7393, 7411, 7417, 7433, 7451,\\ 7457, 7459, 7477, 7481, 7487, 7489, 7507, 7523, 7529, 7537, 7541, 7547, 7549, 7559, 7561, 7573, 7583, 7589, 7591, 7603, 7621, 7639, 7669, 7673, 7681, 7687, 7699, 7703, 7717, 7723, 7741, 7753, 7757, 7759, 7789, 7793, 7829, 7841, 7853, 7867, 7873, 7877, 7879, 7883, 7901, 7907, 7919, 7921, 7927, 7933, 7937, 7949, 7951, 7963, 7993, 8009, 8011, 8017, 8039, 8053, 8059, 8081, 8087, 8089, 8101, 8111, 8117, 8147, 8161, 8167, 8171, 8179, 8191, 8209, 8219, 8221, 8231, 8233, 8243, 8263, 8269, 8287, 8291, 8293, 8297, 8311, 8317, 8329, 8353, 8369, 8377, 8389, 8419, 8423, 8429, 8431, 8443, 8461, 8467, 8513, 8521, 8527, 8537, 8539, 8563, 8573, 8581, 8597, 8599, 8623, 8629, 8641, 8647, 8663, 8677, 8681, 8689, 8693, 8707, 8713, 8719, 8731, 8737, 8741, 8747, 8761, 8779, 8803, 8821, 8831, 8837, 8839, 8849, 8861, 8863, 8867, 8887, 8893, 8923, 8929, 8933, 8941, 8951, 8963, 8971, 8999, 9001, 9007, 9011, 9013, 9029, 9041, 9043, 9049, 9067, 9091, 9103, 9109, 9127, 9133, 9137, 9151, 9157, 9161, 9181, 9187, 9199, 9221, 9239, 9241, 9277, 9281, 9283, 9293, 9311, 9319, 9337, 9341, 9343, 9349, 9371, 9377, 9391, 9397, 9403, 9409, 9419, 9421, 9431, 9433, 9439, 9461, 9463, 9467, 9473, 9479, 9491, 9497, 9511, 9521, 9539, 9547, 9551, 9587, 9601, 9613, 9619, 9629, 9631, 9643, 9649, 9661, 9679, 9697, 9721, 9733, 9739, 9767, 9769, 9781, 9787, 9791, 9811, 9817, 9829, 9839, 9851, 9857, 9859, 9871, 9883, 9901, 9907, 9923, 9929, 9931, 9941, 9949, 9967, 9973, 10007, 10009, 10037, 10039, 10061, 10069,  10091, 10093, 10099, 10111, 10133, 10141, 10151, 10159, 10169, 10177, 10181, 10193, 10201, 10211, 10243, 10259, 10267, 10271, 10273, 10301, 10303, 10313, 10321, 10331, 10333, 10343, 10357, 10369, 10391, 10399, 10427, 10429, 10453, 10459, 10477, 10501, 10513, 10529, 10531, 10567, 10589, 10597, 10601, 10607, 10609, 10627, 10639, 10651, 10657, 10663, 10687, 10691, 10709, 10711, 10723, 10729, 10733, 10739, 10753, 10771, 10781, 10789, 10799, 10831, 10837, 10847, 10859, 10861, 10867, 10891, 10903, 10909, 10939, 10949, 10957, 10973, 10979, 10987, 10993, 11027, 11047, 11059, 11069, 11071, 11083, 11113, 11119, 11131, 11149, 11159, 11161, 11173, 11197, 11239, 11243, 11251, 11257, 11287, 11299, 11311, 11317, 11321, 11329, 11351, 11353, 11369, 11383, 11393, 11411, 11437, 11443, 11447, 11449, 11467, 11491, 11503, 11519, 11527, 11551, 11579, 11587, 11593, 11617, 11657, 11677, 11681, 11689, 11699, 11701, 11717, 11719, 11731, 11743, 11779, 11813, 11821, 11827, 11831, 11833, 11839, 11863, 11867, 11881, 11887, 11903, 11923, 11939, 11941, 11953, 11969, 11971, 12007, 12037, 12041, 12043, 12049, 12071, 12073, 12097, 12101, 12109, 12119, 12157, 12163, 12167, 12203, 12211, 12239, 12241, 12253, 12277, 12281, 12289, 12301, 12329, 12343, 12347, 12373, 12377, 12379, 12391, 12401, 12409, 12413, 12421, 12433, 12451, 12457, 12487, 12491, 12497, 12503, 12511, 12517, 12527, 12541, 12547, 12553, 12577, 12583, 12589, 12601, 12611, 12613, 12619, 12637, 12641, 12653, 12671, 12697, 12703, 12721, 12739, 12743, 12757, 12763, 12769, 12781, 12791, 12799, 12821, 12823, 12829, 12841, 12853, 12889, 12893, 12907, 12919, 12923, 12959, 12967, 12973, 12979, 12983, 13001, 13003, 13009, 13033, 13049, 13063, 13093, 13099, 13121, 13147, 13159, 13171, 13177, 13183, 13219, 13241, 13249, 13267, 13291, 13297, 13309, 13327, 13339, 13367, 13381, 13399, 13411, 13417, 13421, 13441, 13451, 13463, 13477, 13499, 13513, 13537, 13567, 13591, 13597, 13627, 13633, 13669, 13679, 13681, 13687, 13693, 13711, 13723, 13729, 13751, 13759, 13789, 13799, 13807, 13829, 13831, 13841, 13873, 13879, 13903, 13913, 13921, 13931, 13933, 13963, 13967, 14009, 14011, 14029, 14051, 14057, 14071, 14081, 14083, 14107, 14143, 14149, 14173, 14197, 14221, 14251, 14281, 14293, 14323, 14327, 14341, 14347, 14387, 14389, 14401, 14407, 14411, 14419, 14431, 14437, 14449, 14461, 14479, 14503, 14533, 14543, 14551, 14557, 14561, 14563, 14621, 14627, 14641, 14653, 14669, 14683, 14723, 14731, 14737, 14741, 14767, 14779, 14797, 14813, 14821, 14827, 14831, 14851, 14869, 14887, 14891, 14923, 14929, 14939, 14947, 14957, 14983, 15013, 15031, 15061, 15073, 15091, 15101, 15121, 15131, 15139, 15161, 15187, 15193, 15199, 15217, 15241, 15259, 15269, 15271, 15277, 15289, 15299, 15307, 15313, 15319, 15331, 15349, 15361, 15373, 15391, 15401, 15427, 15439, 15443, 15451, 15493, 15511, 15541, 15551, 15559, 15581, 15583, 15601, 15607, 15619, 15625, 15641, 15643, 15649, 15661, 15667, 15679, 15727, 15731, 15733, 15739, 15761, 15787, 15817, 15823, 15859, 15877, 15889, 15901, 15907, 15913, 15919, 15971, 15973, 15991, 16001, 16033, 16057, 16061, 16069, 16087, 16091, 16097, 16111, 16129, 16141, 16183, 16189, 16231, 16249, 16267, 16273, 16333, 16339, 16361, 16363, 16369, 16381, 16411, 16417, 16433, 16447, 16451, 16453, 16477, 16487, 16519, 16561, 16567, 16573, 16603, 16631, 16633, 16651, 16657, 16661, 16693, 16699, 16729, 16741, 16747, 16759, 16807, 16811, 16831, 16843, 16871, 16879, 16883, 16903, 16921, 16927, 16963, 16981, 16987, 16993, 17011, 17021, 17029, 17041, 17047, 17053, 17077, 17107, 17117, 17137, 17161, 17191, 17203, 17209, 17231, 17239, 17257, 17291, 17293, 17299, 17317, 17341, 17351, 17377, 17383, 17389, 17401, 17419, 17431, 17443, 17449, 17467, 17491, 17497, 17509, 17539, 17551, 17569, 17579, 17581, 17599, 17623, 17659, 17683, 17707, 17713, 17737, 17749, 17761, 17791, 17827, 17839, 17851, 17881, 17911, 17923, 17929, 17959, 17971, 17977, 17981, 18013, 18043, 18047, 18049, 18061, 18097, 18121, 18127, 18131, 18133, 18169, 18181, 18191, 18199, 18211, 18217, 18223, 18229, 18251, 18253, 18287, 18289, 18301, 18307, 18313, 18341, 18371, 18379, 18397, 18433, 18439, 18451, 18481, 18493, 18503, 18517, 18521, 18523, 18541, 18553, 18583, 18593, 18637, 18661, 18671, 18679, 18691, 18701, 18719, 18757, 18769, 18793, 18859, 18911, 18919, 18973, 18979, 19009, 19051, 19069, 19081, 19087, 19121, 19139, 19141, 19163, 19183, 19207, 19213, 19219, 19231, 19237, 19267, 19273, 19301, 19309, 19321, 19333, 19381, 19387, 19423, 19429, 19441, 19447, 19471, 19477, 19483, 19489, 19501, 19507, 19531, 19597, 19603, 19609, 19681, 19699, 19717, 19759, 19763, 19777, 19801, 19813, 19819, 19861, 19867, 19891, 19927, 19961, 19963, 19993, 20011, 20021, 20023, 20029, 20051, 20071, 20089, 20101, 20107, 20113, 20143, 20149, 20161, 20173, 20231, 20233, 20269, 20287, 20323, 20341, 20347, 20353, 20359, 20407, 20431, 20479, 20521, 20533, 20551, 20563, 20593, 20611, 20641, 20681, 20707, 20719, 20731, 20749, 20771, 20773, 20809, 20857, 20879, 20887, 20899, 20939, 20947, 20959, 20983, 21001, 21011, 21013,\\ 21023, 21031, 21061, 21067, 21121, 21139, 21163, 21169, 21187, 21191, 21193, 21211, 21221, 21247, 21277, 21313, 21319, 21379, 21391, 21401, 21433, 21481, 21487, 21493, 21517, 21529, 21589, 21601, 21613, 21649, 21661, 21673, 21727, 21739, 21751, 21757, 21799, 21817, 21841, 21851, 21871, 21893, 21911, 21937, 21943, 21961, 21991, 21997, 22003, 22027, 22031, 22051, 22063, 22093, 22111, 22123, 22129, 22147, 22159, 22171, 22189, 22201, 22247, 22271, 22273, 22279, 22283, 22291, 22303, 22369, 22381, 22441, 22447, 22481, 22483, 22501, 22531, 22541, 22543, 22549, 22573, 22621, 22639, 22651, 22699, 22717, 22741, 22751, 22801, 22807, 22861, 22921, 22961, 22963, 23011, 23017, 23041, 23053, 23059, 23071, 23131, 23143, 23167, 23201, 23203, 23209, 23227, 23251, 23269, 23291, 23293, 23311, 23321, 23371, 23431, 23473, 23497, 23509, 23531, 23539, 23557, 23563, 23581, 23599, 23623, 23629, 23669, 23671, 23689, 23719, 23741, 23743, 23761, 23767, 23773, 23801, 23833, 23857, 23869, 23887, 23893, 23899, 23911, 23971, 23977, 24001, 24007, 24019, 24049, 24061, 24091, 24103, 24109, 24121, 24133, 24151, 24169, 24181, 24223, 24229, 24247, 24337, 24371, 24373, 24391, 24421, 24439, 24469, 24481, 24499, 24517, 24551, 24571, 24631, 24649, 24691, 24697, 24709, 24733, 24763, 24781, 24793, 24841, 24847, 24859, 24877, 24889, 24907, 24943, 24967, 24971, 25033, 25057, 25087, 25111, 25117, 25147, 25153, 25171, 25183, 25219, 25229, 25237, 25243, 25301, 25309, 25321, 25391, 25411, 25447, 25453, 25471, 25537, 25561, 25579, 25603, 25621, 25633, 25657, 25693, 25741, 25747, 25771, 25801, 25819, 25841, 25849, 25867, 25873, 25903, 25939, 25951, 25981, 25999, 26021, 26029, 26041, 26053, 26083, 26107, 26113, 26119, 26161, 26203, 26209, 26227, 26251, 26261, 26263, 26293, 26317, 26321, 26371, 26407, 26431, 26479, 26497, 26539, 26569, 26641, 26647, 26681, 26701, 26711, 26713, 26731, 26737, 26821, 26833, 26839, 26861, 26863, 26881, 26893, 26947, 26951, 26959, 27031, 27043, 27061, 27067, 27073, 27091, 27127, 27211, 27241, 27253, 27259, 27271, 27281, 27283, 27337, 27361, 27367, 27397, 27409, 27427, 27457, 27481, 27487, 27541, 27581, 27611, 27631, 27673, 27691, 27739, 27751, 27763, 27799, 27847, 27883, 27889, 27901, 27919, 27941, 27961, 27967, 28051, 28057, 28081, 28099, 28111, 28123, 28183, 28201, 28211, 28219, 28279, 28297, 28309, 28351, 28387, 28393, 28411, 28429, 28447, 28477, 28513, 28549, 28561, 28579, 28591, 28597, 28603, 28621, 28627, 28631, 28657, 28663, 28669, 28687, 28711, 28729, 28751, 28753, 28771, 28813, 28837, 28843, 28867, 28909, 28921, 28927, 29017, 29101, 29131, 29173, 29191, 29201, 29221, 29251, 29269, 29287, 29311, 29383, 29389, 29401, 29411, 29437, 29443, 29527, 29537, 29569, 29581, 29611, 29641, 29671, 29683, 29761, 29791, 29803, 29851, 29863, 29881, 29917, 29929, 29947, 29983, 29989, 30013, 30091, 30097, 30109, 30133, 30139, 30169, 30181, 30187, 30211, 30223, 30241, 30259, 30271, 30307, 30313, 30319, 30367, 30391, 30403, 30493, 30517, 30529, 30559, 30577, 30631, 30637, 30643, 30661, 30697, 30727, 30757, 30781, 30829, 30841, 30853, 30871, 30881, 30937, 30941, 30949, 31033, 31039, 31051, 31063, 31069, 31081, 31123, 31147, 31159, 31177, 31183, 31189, 31231, 31249, 31271, 31321, 31333, 31357, 31393, 31489, 31513, 31531, 31567, 31573, 31627, 31663, 31741, 31771, 31873, 31891, 31957, 31963, 31981, 32029, 32041, 32059, 32077, 32089, 32191, 32203, 32233, 32251, 32257, 32299, 32323, 32341, 32371, 32377, 32401, 32491, 32533, 32561, 32563, 32587, 32611, 32653, 32707, 32713, 32719, 32761, 32779, 32797, 32803, 32833, 32839, 32869, 32887, 32911, 32941, 32971, 33013, 33049, 33091, 33151, 33181, 33199, 33211, 33223, 33247, 33289, 33301, 33331, 33343, 33349, 33391, 33409, 33427, 33457, 33529, 33589, 33601, 33637, 33679, 33721, 33751, 33769, 33791, 33811, 33931, 33961, 34039, 34057, 34123, 34129, 34141, 34147, 34171, 34183, 34211, 34231, 34261, 34273, 34297, 34327, 34351, 34381, 34471, 34501, 34511, 34519, 34549, 34591, 34603, 34651, 34687, 34693, 34759, 34819, 34843, 34849, 34939, 34981, 35023, 35053, 35083, 35149, 35221, 35227, 35251, 35281, 35311, 35317, 35401, 35407, 35461, 35491, 35509, 35527, 35533, 35569, 35671, 35731, 35803, 35839, 35851, 35863, 35869, 35911, 35983, 36037, 36061, 36073, 36097, 36109, 36151, 36187, 36241, 36307, 36313, 36341, 36343, 36373, 36433, 36451, 36457, 36481, 36523, 36541, 36559, 36571, 36583, 36637, 36691, 36709, 36721, 36781, 36793, 36847, 36871, 36877, 36901, 36919, 36943, 36973, 36979, 36997, 37021, 37039, 37123, 37171, 37181, 37189, 37201, 37249, 37321, 37339, 37363, 37423, 37441, 37489, 37501, 37507, 37549, 37561, 37567, 37571, 37573, 37591, 37657, 37699, 37717, 37747, 37781, 37783, 37831, 37861, 37879, 37951, 37963, 38011, 38047, 38053, 38083, 38119, 38197, 38281, 38287, 38329, 38371, 38377, 38431, 38461, 38557, 38593, 38611, 38671, 38677, 38707, 38749, 38767, 38791, 38803, 38809, 38821, 38851, 38861, 38917, 38971, 38977, 39043, 39079, 39097, 39139, 39157, 39181, 39217, 39241, 39301, 39313, 39341, 39373, 39439, 39451, 39499, 39511, 39601, 39607, 39619, 39631, 39679, 39727, 39733, 39799, 39841, 39883, 39901, 39979, 40063, 40087, 40093, 40111, 40153, 40177, 40189, 40231, 40357, 40387, 40429, 40471, 40483, 40519, 40531, 40591, 40597, 40609, 40627, 40639, 40693, 40699, 40771, 40801, 40819, 40849, 40903, 40933, 40993, 41011, 41077, 41113, 41131, 41141, 41161, 41203, 41221, 41227, 41281, 41341, 41413, 41479, 41491, 41521, 41539, 41617, 41641, 41647, 41659, 41761, 41851, 41887, 41893, 41911, 41941, 41959, 42043, 42061, 42139, 42181, 42331, 42337, 42373, 42391, 42409, 42433, 42451, 42457, 42463, 42487, 42499, 42571, 42589, 42641, 42643, 42667, 42697, 42709, 42751, 42793, 42841, 42901, 42961, 42967, 42979, 43003, 43051, 43093, 43117, 43177, 43201, 43261, 43291, 43321, 43411, 43451, 43579, 43591, 43597, 43633, 43651, 43669, 43711, 43759, 43801, 43891, 43933, 43951, 43963, 44029, 44041, 44053, 44071, 44101, 44221, 44263, 44281, 44293, 44371, 44389, 44521, 44587, 44641, 44683, 44773, 44839, 44851, 44893, 44971, 44983, 45007, 45061, 45181, 45319, 45343, 45361, 45427, 45481, 45523, 45541, 45613, 45631, 45691, 45757, 45763, 45841, 45943, 46021, 46027, 46051, 46061, 46099, 46141, 46153, 46171, 46261, 46309, 46327, 46351, 46411, 46441, 46447, 46471, 46489, 46567, 46687, 46691, 46747, 46751, 46771, 46831, 46861, 46957, 46993, 47017, 47041, 47161, 47221, 47251, 47269, 47287, 47431, 47497, 47521, 47581, 47623, 47629, 47653, 47701, 47713, 47737, 47779, 47791, 47797, 47881, 47911, 47917, 47947, 48049, 48091, 48109, 48259, 48337, 48463, 48571, 48619, 48661, 48673, 48733, 48751, 48757, 48781, 48799, 48847, 48871, 48889,\\ 48907, 48991, 49069, 49081, 49123, 49171, 49177, 49201, 49207, 49261, 49331, 49369, 49393, 49411, 49417, 49429, 49477, 49531, 49633, 49663, 49681, 49711, 49729, 49831, 49843, 49891, 49921, 49939, 50053, 50077, 50101, 50131, 50221, 50287, 50311, 50329, 50341, 50359, 50383, 50527, 50551, 50581, 50593, 50653, 50671, 50773, 50821, 50971, 50989, 51001, 51031, 51151, 51157, 51193, 51241, 51283, 51307, 51349, 51361, 51481, 51511, 51517, 51529, 51613, 51679, 51769, 51853, 51871, 51907, 51913, 51949, 51991, 52009, 52021, 52027, 52081, 52147, 52201, 52237, 52291, 52321, 52361, 52441, 52453, 52501, 52543, 52561, 52579, 52627, 52711, 52861, 52903, 52957, 52981, 53047, 53101, 53173, 53197, 53239, 53281, 53299, 53407, 53551, 53569, 53593, 53629, 53719, 53731, 53791, 53881, 53923, 54001, 54037, 54091, 54121, 54163, 54181, 54217, 54251, 54289, 54331, 54401, 54421, 54469, 54499, 54517, 54541, 54583, 54601, 54631, 54667, 54709, 54721, 54751, 54919, 54973, 54979, 55021, 55147, 55171, 55201, 55291, 55351, 55381, 55411, 55441, 55603, 55621, 55639, 55681, 55711, 55819, 55903, 55927, 55987, 56053, 56101, 56197, 56359, 56377, 56401, 56431, 56473, 56533, 56569, 56611, 56629, 56701, 56731, 56737, 56767, 56809, 56821, 56827, 56893, 56911, 56941, 56989, 57073, 57121, 57241, 57271, 57331, 57367, 57457, 57601, 57709, 57751, 57781, 57853, 57901, 58081, 58099, 58111, 58171, 58231, 58321, 58411, 58543, 58573, 58657, 58711, 58741, 58771, 58789, 58831, 59011, 59021, 59077, 59221, 59263, 59281, 59341, 59419, 59467, 59473, 59497, 59509, 59581, 59621, 59671, 59707, 59791, 59797, 59887, 59971, 60013, 60103, 60127, 60139, 60271, 60331, 60337, 60589, 60601, 60607, 60631, 60661, 60733, 60811, 60859, 60901, 60919, 60937, 61051, 61231, 61261, 61291, 61363, 61381, 61441, 61471, 61561, 61651, 61681, 61723, 61861, 61909, 61987, 62011, 62071, 62119, 62131, 62191, 62311, 62323, 62497, 62533, 62581, 62701, 62731, 62773, 62791, 62869, 62971, 62983, 63001, 63031, 63127, 63211, 63241, 63277, 63331, 63361, 63391, 63487, 63541, 63649, 63667, 63781, 63841, 63901, 64081, 64171, 64231, 64381, 64567, 64579, 64591, 64621, 64747, 64891, 64921, 64927, 64951, 65011, 65053, 65101, 65179, 65269, 65497, 65521, 65581, 65599, 65647, 65701, 65731, 65809, 65881, 65899, 66049, 66067, 66221, 66271, 66301, 66361, 66403, 66529, 66601, 66721, 66751, 66889, 66931, 66943, 67033, 67057, 67141, 67213, 67231, 67411, 67447, 67489, 67531, 67537, 67651, 67699, 67741, 67789, 67891, 68041, 68113, 68161, 68311, 68329, 68371, 68443, 68521, 68581, 68611, 68791, 68821, 68881, 69001, 69031, 69169, 69259, 69427, 69481, 69499, 69661, 69709, 69829, 69931, 70039, 70051, 70111, 70141, 70183, 70201, 70309, 70381, 70489, 70501, 70621, 70687, 70753, 70849, 70891, 70921, 70981, 71171, 71191, 71341, 71353, 71479, 71527, 71551, 71761, 71821, 71941, 71971, 71983, 72031, 72073, 72091, 72109, 72211, 72271, 72307, 72361, 72421, 72559, 72577, 72661, 72739, 72871, 72901, 72907, 72931, 72937, 72997, 73063, 73141, 73189, 73243, 73291, 73327, 73351, 73441, 73459, 73471, 73651, 73693, 73999, 74047, 74071, 74101, 74131, 74161, 74197, 74257, 74377, 74383, 74413, 74449, 74521, 74551, 74611, 74713, 74719, 74731, 74827, 74869, 74941, 75037, 75079, 75133, 75181, 75211, 75391, 75403, 75511, 75541, 75583, 75709, 75721, 75781, 75853, 75931, 76213, 76231, 76423, 76441, 76519, 76561, 76651, 76729, 76771, 77023, 77041, 77191, 77239, 77263, 77317, 77347, 77431, 77491, 77527, 77611, 77617, 77641, 77743, 77977, 78031, 78121, 78157, 78241, 78367, 78541, 78571, 78607, 78691, 78961, 79111, 79153, 79201, 79531, 79561, 79633, 79699, 79801, 79861, 80089, 80173, 80191, 80263, 80317, 80341, 80761, 80911, 81001, 81019, 81181, 81307, 81331, 81421, 81559, 81649, 81901, 82009, 82021, 82051, 82141, 82171, 82189, 82231, 82237, 82261, 82279, 82351, 82531, 82651, 82657, 82699, 82891, 82963, 83071, 83227, 83449, 83521, 83617, 83701, 83791, 83911, 84061, 84121, 84181, 84211, 84319, 84421, 84457, 84631, 84691, 84751, 84871, 84961, 84991, 85021, 85081, 85087, 85159, 85201, 85213, 85411, 85429, 85549, 85621, 85831, 85933, 86011, 86131, 86143, 86293, 86311, 86461, 86851, 86923, 87049, 87103, 87121, 87151, 87187, 87211, 87481, 87511, 87517, 87649, 87691, 87721, 87751, 87931, 87991, 88237, 88321, 88411, 88741, 89041, 89119, 89209, 89317, 89371, 89611, 89821, 89833, 89839, 90001, 90031, 90073, 90121, 90271, 90481, 90511, 90631, 90793, 90847, 90901, 90931, 91081, 91141, 91291, 91387, 91513, 91591, 91711, 91771, 91807, 91909, 91939, 92251, 92269, 92311, 92401, 92431, 92641, 92737, 92821, 92899, 92941, 93151, 93241, 93601, 93811, 93871, 93941, 94117, 94249, 94441, 94531, 94621, 95131, 95239, 95311, 95419, 95461, 95581, 95701, 95881, 95911, 95923, 96331, 96601, 96643, 96697, 96721, 96931, 97021, 97171, 97381, 97441, 97561, 97651, 97813, 97861, 97969, 98011, 98101, 98491, 98641, 98731, 98911, 98947, 99181, 99577, 99661, 99667, 99721, 99793, 99859, 99901, 99991, 100153, 100207, 100279, 100291, 100621, 100741, 100801, 100981, 101161, 101221, 101341, 101581, 101611, 102001, 102061, 102181, 102241, 102547, 102673, 103231, 103357, 103471, 103483, 103561, 103573, 103651, 103681, 103801, 103951, 104281, 104311, 104347, 104491, 104743, 105211, 105277, 105361, 105751, 106219, 106261, 106291, 106531, 106591, 106681, 106759, 106921, 107101, 107227, 107251, 107641, 107713, 107881, 108109, 108271, 108301, 108343, 108439, 108571, 108631, 108991, 109141, 109561, 109621, 109741, 109831, 109891, 110161, 110431, 110581, 110821, 110881, 111091, 111211, 111871, 112291, 112771, 112921, 113023, 113131, 113161, 113341, 113569, 113761, 113779, 114031, 114451, 114571, 114601, 114643, 114661, 115021, 115201, 115831, 116191, 116341, 116461, 116533, 116911, 117361, 117541, 117721, 117811, 117991, 118171, 118471, 118603, 118621, 118801, 118861, 118927, 119191, 119557, 119611, 119701, 119851, 119971, 120103, 120121, 120331, 120409, 120511, 120691, 120871, 121321, 121501, 121531, 121591, 121801, 121951, 122041, 122131, 122599, 122761, 122833, 123631, 123661, 124021, 124147, 124231, 124471, 124951, 125101, 125371, 125551, 125641, 125731, 125791, 126001, 126127, 126541, 126631, 126961, 127051, 127261, 127297, 128221, 128431, 128461, 128521, 128881, 129361, 130051, 130321, 130411, 131041, 131101, 131251, 131371, 131581, 131671, 132241, 132661, 132967, 133201, 133387, 133519, 133723, 133981, 134689, 134731, 135151, 135391, 135433, 135661, 135721, 135781, 136501, 136531, 136621, 136711, 136861, 137251, 137341, 137491, 138007, 138451, 138511, 139129, 139591, 139987, 140071, 140221, 140401, 140611, 140761, 141121, 141241, 141481, 142381, 143551, 143641, 143677, 143821, 144271, 144481, 145471, 145531, 145681, 145861, 146323, 146521, 146701, 146953, 147331, 148171, 148201, 148411,\\ 148471, 148501, 148891, 149689, 149731, 150151, 150301, 150697, 150991, 151141, 151321, 151381, 151531, 151561, 152041, 152083, 152461, 152791, 152821, 153001, 153271, 153871, 154081, 154981, 155521, 156061, 156151, 156241, 156601, 156781, 157081, 157231, 157411, 157543, 157609, 158341, 158761, 158941, 160651, 160801, 161071, 161221, 161281, 161911, 162451, 162691, 164011, 164341, 165601, 165811, 165961, 167281, 167311, 167491, 167521, 168151, 168211, 168499, 168541, 169003, 169471, 169831, 170101, 170641, 170881, 171091, 171271, 172279, 172981, 173191, 174241, 174721, 174931, 175561, 176023, 176551, 176887, 177841, 178831, 179281, 179461, 180181, 180391, 182701, 185131, 185221, 185641, 185761, 186391, 186481, 187471, 187489, 187651, 188701, 188911, 191251, 192259, 192271, 192721, 192781, 194371, 194581, 194671, 196561, 197341, 198901, 199411, 200341, 200971, 201601, 201961, 202231, 203221, 204751, 204931, 205111, 207307, 208261, 208891, 210421, 211051, 212521, 212671, 212851, 213481, 214831, 216217, 216901, 217081, 218611, 220411, 220861, 221671, 223381, 224011, 224071, 224401, 229681, 229771, 230311, 231331, 231481, 232381, 233641, 234271, 235621, 235951, 238681, 240571, 241081, 244351, 247501, 249211, 252541, 255511, 263071, 267961, 268291, 270271, 271261, 271441, 275881, 279001, 279841, 283861, 284131, 286651, 292231, 295831, 296011, 300301, 300511, 302941, 310501, 310591, 312841, 314161, 318691, 320401, 321031, 322921, 326041, 333271, 337411, 343261, 347131, 358801, 361351, 376471, 426871, 434281, 434611, 440911, 493021, 516961, 656371, 673201\\
\hline	
\end{longtable}}

\bibliographystyle{unsrt}
\bibliography{bib.bib}

\begin{thebibliography}{10}

\bibitem{huc13}
Sophie Huczynska, Gary~L Mullen, Daniel Panario, and David Thomson.
\newblock Existence and properties of k-normal elements over finite fields.
\newblock {\em Finite Fields and Their Applications}, 24:170--183, 2013.

\bibitem{Lenstra}
H.~W. Lenstra, Jr. and R.~J. Schoof.
\newblock Primitive normal bases for finite fields.
\newblock {\em Math. Comp.}, 48(177):217--231, 1987.

\bibitem{cao2014prim}
Xiwang Cao and Peipei Wang.
\newblock Primitive elements with prescribed trace.
\newblock {\em Applicable Algebra in Engineering, Communication and Computing},
  25:339--345, 2014.

\bibitem{cohen2006strong}
Stephen~D Cohen and Sophie Huczynska.
\newblock The strong primitive normal basis theorem.
\newblock {\em arXiv preprint math/0610400}, 2006.

\bibitem{gauss}
Stephen~D Cohen.
\newblock Gauss sums and a sieve for generators of galois fields.
\newblock {\em Publ. Math. Debrecen}, 56(2-3):293--312, 2000.

\bibitem{kapetanakis2019variations}
Giorgos Kapetanakis and Lucas Reis.
\newblock Variations of the primitive normal basis theorem.
\newblock {\em Designs, Codes and Cryptography}, 87(7):1459--1480, 2019.

\bibitem{Avnish1}
Avnish~K Sharma, Mamta Rani, and Sharwan~K Tiwari.
\newblock Primitive normal pairs with prescribed norm and trace.
\newblock {\em Finite Fields and Their Applications}, 78:101976, 2022.

\bibitem{Avnish2}
Avnish~K Sharma, Mamta Rani, and Sharwan~K Tiwari.
\newblock Primitive normal values of rational functions over finite fields.
\newblock {\em Journal of Algebra and Its Applications}, page 2350152, 2022.

\bibitem{Rani1}
Mamta Rani, Avnish~K Sharma, Sharwan~K Tiwari, and Indivar Gupta.
\newblock On the existence of pairs of primitive normal elements over finite
  fields.
\newblock {\em Sao Paulo Journal of Mathematical Sciences}, 16(2):1032--1049,
  2022.

\bibitem{fan1}
Shuqin Fan and Wenbao Han.
\newblock $p$-adic formal series and primitive polynomials over finite fields.
\newblock {\em Proceedings of the American Mathematical Society},
  132(1):15--31, 2004.

\bibitem{fan2}
Fan Shuqin and Han Wenbao.
\newblock Primitive polynomials over finite fields of characteristic two.
\newblock {\em Applicable Algebra in Engineering, Communication and Computing},
  14:381--395, 2004.

\bibitem{cohenprimpoly}
Stephen~D Cohen.
\newblock Primitive polynomials with a prescribed coefficient.
\newblock {\em Finite Fields and Their Applications}, 12(3):425--491, 2006.

\bibitem{cohenhansen}
S~Cohen and Mateja Presern.
\newblock The hansen-mullen primitivity conjecture: completion of proof.
\newblock {\em London Mathematical Society Lecture Note Series}, 352:89, 2008.

\bibitem{fan3}
Fan Shuqin and Han Wenbao.
\newblock $p$-adic formal series and cohen's problem.
\newblock {\em Glasgow Mathematical Journal}, 46(1):47--61, 2004.

\bibitem{fan4}
Shuqin Fan, Wenbao Han, and Keqin Feng.
\newblock Primitive normal polynomials with multiple coefficients prescribed:
  An asymptotic result.
\newblock {\em Finite Fields and Their Applications}, 13(4):1029--1044, 2007.

\bibitem{fan5}
Shuqin Fan and Xiaozhe Wang.
\newblock Primitive normal polynomials with a prescribed coefficient.
\newblock {\em Finite Fields and Their Applications}, 15(6):682--730, 2009.

\bibitem{fan6}
Shuqin Fan.
\newblock Primitive normal polynomials with the last half coefficients
  prescribed.
\newblock {\em Finite Fields and Their Applications}, 15(5):604--614, 2009.

\bibitem{fan7}
Shuqin Fan and Xiaozhe Wang.
\newblock Primitive normal polynomials with the specified last two
  coefficients.
\newblock {\em Discrete mathematics}, 309(13):4502--4513, 2009.

\bibitem{morgan}
Ilene~H Morgan and Gary~L Mullen.
\newblock Primitive normal polynomials over finite fields.
\newblock {\em Mathematics of computation}, 63(208):759--765, 1994.

\bibitem{negre}
Christophe Negre.
\newblock Finite field arithmetic using quasi-normal bases.
\newblock {\em Finite Fields and Their Applications}, 13(3):635--647, 2007.

\bibitem{Rani2}
Mamta Rani, Avnish~K Sharma, and Sharwan~K Tiwari.
\newblock On $r$-primitive $k$-normal elements over finite fields.
\newblock {\em Finite Fields and Their Applications}, 82:102053, 2022.

\bibitem{Rani3}
Mamta Rani, Avnish~K Sharma, Sharwan~K Tiwari, and Anupama Panigrahi.
\newblock Inverses of $r$-primitive $k$-normal elements over finite fields.
\newblock {\em The Ramanujan Journal}, pages 1--25, 2023.

\bibitem{Rani4}
Mamta Rani, Avnish~K. Sharma, Sharwan~K. Tiwari, and Anupama Panigrahi.
\newblock On $r$-primitive $k$-normal elements with prescribed norm and trace
  over finite fields.
\newblock {\em Finite Fields and Their Applications}, 91:102253, 2023.

\bibitem{reis19}
Lucas Reis.
\newblock Existence results on $ k $-normal elements over finite fields.
\newblock {\em Revista Matem{\'a}tica Iberoamericana}, 35(3):805--822, 2019.

\bibitem{Nieder}
Rudolf Lidl and Harald Niederreiter.
\newblock {\em Finite fields}, volume~20.
\newblock Cambridge University Press, Cambridge, second edition, 1997.

\bibitem{schmidt}
Wolfgang~M Schmidt.
\newblock {\em Equations over finite fields: an elementary approach}, volume
  536.
\newblock Springer Berlin, Heidelberg, 2006.

\bibitem{RK}
Josimar~JR Aguirre, C{\'\i}cero Carvalho, and Victor~GL Neumann.
\newblock About $r$-primitive and $k$-normal elements in finite fields.
\newblock {\em Designs, Codes and Cryptography}, 91(1):115--126, 2023.

\bibitem{han96}
Wen~Bao Han.
\newblock The coefficients of primitive polynomials over finite fields.
\newblock {\em Mathematics of computation}, 65(213):331--340, 1996.

\bibitem{CoHucWc}
Stephen~D. Cohen and Sophie Huczynska.
\newblock The primitive normal basis theorem---without a computer.
\newblock {\em J. London Math. Soc. (2)}, 67(1):41--56, 2003.

\bibitem{CoHuc}
Stephen~D. Cohen and Sophie Huczynska.
\newblock The strong primitive normal basis theorem.
\newblock {\em Acta Arith.}, 143(4):299--332, 2010.

\bibitem{KapeNBT}
Giorgos Kapetanakis.
\newblock An extension of the (strong) primitive normal basis theorem.
\newblock {\em Appl. Algebra Engrg. Comm. Comput.}, 25(5):311--337, 2014.

\bibitem{sagemath}
{The Sage Developers}.
\newblock {\em {S}ageMath, the {S}age {M}athematics {S}oftware {S}ystem
  ({V}ersion 9.0)}, 2020.

\end{thebibliography}
\end{document}